\documentclass{amsart}

\usepackage{amsmath}
\usepackage{amssymb}
\usepackage{amsfonts}
\usepackage{amsthm}
\usepackage{latexsym}
\usepackage{esint}
\usepackage{mathtools}
\usepackage{mathrsfs}
\usepackage{graphicx}
\usepackage{bbm}
\usepackage{color}

\newtheorem{theorem}{Theorem}[section]
\newtheorem{lemma}[theorem]{Lemma}
\newtheorem{proposition}[theorem]{Proposition}
\newtheorem{corollary}[theorem]{Corollary}

\theoremstyle{definition}
\newtheorem{definition}[theorem]{Definition}

\theoremstyle{remark}
\newtheorem{remark}[theorem]{Remark}

\numberwithin{equation}{section}

\begin{document}

\title[Existence and low-Mach limit for stochastic compressible flows]{EXISTENCE OF MARTINGALE SOLUTIONS AND THE
INCOMPRESSIBLE LIMIT FOR STOCHASTIC COMPRESSIBLE
FLOWS ON THE WHOLE SPACE}

\author{Prince Romeo Mensah}
\address{Department of Mathematics, Heriot-Watt University, Edinburgh, EH14 4AS, United Kingdom}
\email{pm27@hw.ac.uk}
\thanks{The author would like to acknowledge the financial support of the Department of Mathematics, Heriot--Watt University, through the James--Watt scholarship. He will also like to thank D. Breit for recommending this work and for his many useful discussions.}



\date{}



\begin{abstract}
We give an existence and asymptotic result for the so-called finite energy weak martingale solution of the compressible isentropic Navier--Stokes system driven by some random force in the whole spatial region. In particular, given a general nonlinear multiplicative noise, we establish the convergence to the incompressible system as the Mach number, representing the ratio between the average flow velocity and the speed of sound, approaches zero. 
\end{abstract}

\maketitle

\section{Introduction}
\label{ch:introduction}

In continuum mechanics, the motion of an \textit{isentropic} compressible fluid is described by the density $\varrho=\varrho(t, x)$ and velocity $\mathbf{u}=\mathbf{u}(t , x)$ in a physical domain in $\mathbb R^3$ satisfying the \textit{mass and momentum balance equations} given respectively by
\begin{equation}
\begin{aligned}
\label{massAndMomentum}
\partial_t\varrho + \mathrm{div}(\varrho \mathbf{u}) = 0,
\\
\partial_t(\varrho \mathbf{u}) + \mathrm{div}(\varrho \mathbf{u} \otimes \mathbf{u}) = \mathrm{div}\mathbb{T} + \varrho \mathbf{f}.
\end{aligned}
\end{equation}
Here $\mathbf{f}$ is some external force and $\mathbb{T}$ the \textit{stress tensor}. By  \textit{Stokes' law}, $\mathbb{T}$ satisfies $\mathbb{T}= \mathbb{S}- p\mathbb{I}$ where $p=p(\varrho)$ is the pressure and  $\mathbb{S}=\mathbb{S}(\nabla\mathbf{u})$ the \textit{viscous stress tensor}. In following \textit{Newton's law of viscosity}, we assume that $\mathbb{S}$ satisfies 
\begin{align*}
\mathbb{S}= \nu\left(\nabla \mathbf{u} + \nabla^T\mathbf{u}  \right) + \lambda\,\mathrm{div}\,\mathbf{u}\mathbb{I}
\end{align*}  
with viscosity coefficients satisfying $\nu>0$, $ \lambda + \frac{2}{3}\nu\geq 0.$
For the pressure, we suppose the $\gamma$-law
\begin{align*}
p=\frac{1}{\mathrm{Ma}^2}\varrho^\gamma
\end{align*}
where $\mathrm{Ma}>0$ is the Mach number and $\gamma>\frac{3}{2}$, the adiabatic exponent. In order to study the existence of solutions
to system \eqref{massAndMomentum}, it has to be complemented by initial and boundary conditions (very common are periodic boundary conditions, no-slip boundary conditions and the whole space). The existence of weak solutions to  \eqref{massAndMomentum} has been shown in the fundamental book by Lions
\cite{lions1998mathematical}
and extended to physical reasonable situations by Feireisl \cite{feireisl2001compactness, feireisl2001existence},
giving a compressible analogue of the pioneering work by Leray \cite{leray1934mouvement} on the incompressible case. These results involve the concept of \textit{weak solutions} where derivatives have to be understood in the \textit{sense of distributions}. This concept has since become an integral technique in the study of nonlinear PDE's.

In recent years, there has been an increasing interest in random influences
on fluid motions. It can take into account, for example, physical, empirical or numerical uncertainties and is commonly used to model turbulence in the fluid motion.

As far as we know, the first result on the existence of solution to the stochastic compressible system is due to \cite{tornatore1997one}. This was done in 1-D and later for a special periodic 2-D case in \cite{tornatore2000global}. The latter mostly relied on existence arguments developed in \cite{vaigant1995existence}. In \cite{feireisl2013compressible}, a semi-deterministic approach based on results on multi-valued functions is used and follows in line with the incompressible analogue shown in \cite{bensoussan1973equations}. A fully stochastic theory has been developed in \cite{Hof}. The existence of martingale solutions has been shown in the case of periodic boundary conditions. This has been extended to Dirichlet boundary conditions in \cite{smith2015random}.

Compared to the stochastic compressible model, the incompressible system has been studied much  more intensively. It  first appeared in the seminal paper by Bensoussan and Temam \cite{bensoussan1973equations} which is based on a semi-deterministic approach. Later, the concept of a martingale solution of this system  was then introduced by Flandoli and Gatarek  \cite{flandoli1995martingale}. For a recent survey on the stochastic incompressible Navier--Stokes equations, we refer the reader to \cite{romito2016probabilistic} or to \cite{robinson2016recent} for the general survey including deterministic results.

The aim of this paper is to look at the situation on the whole space $\mathbb R^3$. This is particularly important for various applications and especially for those in which  the comparative size of the fluids domain far exceeds  the  speed of sound accompanying the fluid. See \cite{feireisl2009singular} for more details. Difficulties arise due to the lack of certain compactness tools which are available in the case of bounded domains. We shall study the system
\begin{equation}
\begin{aligned}
\label{comprSPDE0}
\mathrm{d}\varrho + \mathrm{div}(\varrho \mathbf{u})\mathrm{d}t = 0, \\
\mathrm{d}(\varrho \mathbf{u}) + [\mathrm{div}(\varrho \mathbf{u}\otimes \mathbf{u}-\mathbb S(\nabla\mathbf{u})) + \nabla p(\varrho)]\mathrm{d}t = \Phi(\varrho,\varrho \mathbf{u})\mathrm{d}W,
\end{aligned}
\end{equation}
in $Q_T=(0,T)\times\mathbb R^3$. A prototype
for the stochastic forcing term will be given by
\begin{equation}
\label{diffCoeciffients}
\Phi(\varrho,\varrho \mathbf{u})\mathrm{d}W \approx \varrho\mathrm{d}W^1 + \varrho \mathbf{u}\mathrm{d}W^2
\end{equation}
where   $W^1$ and $W^2$ is a pair of independent cylindrical Wiener processes. We refer to Sect. \ref{Preliminaries} for the precise assumptions on the noise and its coefficients.

The first main result of the present paper is the existence of finite energy weak martingale solutions to \eqref{comprSPDE0}. The precise statement is given in Theorem \ref{thm:dissi}. 
We approximate the system on the whole space by a sequence of periodic problems (where the period tends to infinity). After showing uniform a priori estimates, we use the stochastic compactness method based on the Jakubowski-Skorokhod representation theorem. In contrast to previous works, we adapt it to the situation on the whole space taking carefully into account, the lack of compact embeddings. In order to pass to the limit in the nonlinear pressure term, we use properties of the effective viscous flux originally introduced by Lion \cite{lions1998mathematical} similar to \cite{Hof}. 

A fundamental question in compressible fluid mechanics is the relation to the incompressible model. If the Mach number is small, the fluid should behave asymptotically like an incompressible one, provided velocity and viscosity are small, and we are looking at large time scales, see \cite{klainerman1981singular}.
The problem has been studied rigorously in the deterministic case in \cite{lions1998incompressible, lions1998unicite, lions1999approche}, as a singular limit problem. A major problem to overcome is the rapid oscillation of acoustic waves due to the lack of compactness. A stochastic counterpart of this theory has very recently been established in \cite{breit2015incompressible}. The limit $\varepsilon$ of the system
\begin{equation}
\begin{aligned}
\label{comprSPDE}
\mathrm{d}\varrho_\varepsilon + \mathrm{div}(\varrho_\varepsilon \mathbf{u}_\varepsilon)\mathrm{d}t = 0, \\
\mathrm{d}(\varrho_\varepsilon \mathbf{u}_\varepsilon) + [\mathrm{div}(\varrho_\varepsilon \mathbf{u}_\varepsilon\otimes \mathbf{u}_\varepsilon-\mathbb S(\nabla\mathbf{u}_\varepsilon)) + \nabla \frac{\varrho^\gamma_\varepsilon}{\varepsilon^2}]\mathrm{d}t = \Phi(\varrho_\varepsilon,\varrho_\varepsilon \mathbf{u}_\varepsilon)\mathrm{d}W,
\end{aligned}
\end{equation}
has been analyzed under periodic boundary conditions. Given a  sequence of the so-called \textit{finite energy weak martingale solution} for \eqref{comprSPDE} (see next section for definition) where $\varepsilon\in(0,1)$ ,  its limit (as $\varepsilon\rightarrow 0$) is indeed a \textit{weak martingale solution} to the following incompressible system: 
\begin{equation}
\begin{aligned}
\label{incomprSPDE}
\mathrm{div}( \mathbf{u}) = 0, \\
\mathrm{d}( \mathbf{u}) + [\mathrm{div}( \mathbf{u}\otimes \mathbf{u})-\nu\Delta \mathbf{u}  +\nabla \tilde{p}]\mathrm{d}t = \mathcal{P}\Phi(1, \mathbf{u})\mathrm{d}W.
\end{aligned}
\end{equation}
Here $\tilde{p}$ is the associated pressure and $\mathcal{P}$ is the Helmholtz projection onto the space of solenoidal vector fields.

A major drawback in the approach in \cite{breit2015incompressible} is that the noise coefficient $\Phi(\varrho,\varrho\mathbf{u})$ has to be linear in the momentum $\varrho\mathbf{u}$. This is due to the aforementioned lack of compactness of momentum when $\varepsilon$ passes to zero. This cannot even be improved in the deterministic case. The situation on the whole space, however, is much better as a consequence of dispersive estimates for the acoustic wave equations, see Proposition \ref{gradientPartOfMomentumConvergece}. We apply them 
to the stochastic wave equation and hence are  able to prove strong convergence of the momentum, see Lemma \ref{momentumStrong}.
Based on this, we are able to prove the convergence of \eqref{comprSPDE} to \eqref{incomprSPDE} under much more general assumptions on the noise coefficients. See Theorem \ref{thm:one} for details.

In Sect. \ref{Preliminaries}, we state the required assumptions satisfied by the various quantities used in this paper, as well as some useful function space estimates.
We define the concept of a solution, state the required boundary condition applicable in our setting and finally state the main results.

In Sect.  \ref{existence}, we are concerned with  the proof of Theorem \ref{thm:dissi}, giving existence of martingale solutions on the whole space. Based on this result, we devote Sect. \ref{singularLimit} to the proof of Theorem \ref{thm:one}; the low-Mach number limit on the whole space.

\section{Preliminaries}
\label{Preliminaries}
Throughout this paper, the spatial dimension is $N=3$ and we assume that $(\Omega,\mathscr{F},(\mathscr{F}_t)_{t\geq0},\mathbb{P})$ is a stochastic basis with a complete right-continuous filtration,
$W$ is a $(\mathscr{F}_t)$-cylindrical Wiener process, that is, there exists a family of mutually independent real-valued Brownian motions $(\beta_k)_{k\in\mathbb{N}}$ and orthonormal basis $(e_k)_{k\in\mathbb{N}}$ of a separable Hilbert space $\mathfrak{U}$ such that
\begin{align*}
W(t) =  \sum_{k\in\mathbb{N}}\beta_k(t)e_k, \quad t\in[0,T].
\end{align*}
We also assume that $\varrho  \in  L^\gamma_{\mathrm{loc}}(\mathbb{R}^3)$, $\varrho\geq 0$, and $\mathbf{u}\in L^2_{\mathrm{loc}}(\mathbb{R}^3)$ so that $\sqrt{\varrho}\mathbf{u} \in L^2_{\mathrm{loc}}(\mathbb{R}^3)$.

Now let set $\mathbf{q}=\varrho\mathbf{u}$ and assume that there exists a compact set $\mathcal{K}\subset\mathbb{R}^3$ and some functions $g_k: \mathbb{R}^3\times \mathbb{R}\times\mathbb{R}^3  \rightarrow \mathbb{R}^3$  such that 
\begin{align}
\label{noiseSupport}
g_k\in C^1_0 \left(\mathcal{K}\right), \quad \text{for any } k\in\mathbb{N},
\end{align}
and in addition, satisfies the following growth conditions:
\begin{equation}
\begin{aligned}
\label{stochCoeffBound}
\sum_{k\in\mathbb{N}}  \vert g_k(x, \varrho,  \mathbf{q})   \vert^2   \leq c\,\left( \varrho^2 + \vert \mathbf{q}\vert^2\right),
\quad
\sum_{k\in\mathbb{N}}  \vert\nabla_{\varrho,\mathbf{q}} \,g_k(x, \varrho, \mathbf{q}) \vert^2   \leq c.
\end{aligned}
\end{equation}
Then if we define the map $\Phi(\varrho, \varrho \mathbf{u}):\mathfrak{U}\rightarrow    L^1(\mathcal{K})$ by
$\Phi(\varrho, \varrho \mathbf{u}) e_k   =  g_k(\cdot, \varrho(\cdot), \varrho \mathbf{u} (\cdot))$,  we can use   the embedding $L^1(\mathcal{K})\hookrightarrow  W^{-l,2}(\mathcal{K})$ where $l>\frac{3}{2}$, to show that
$
\Vert \Phi(\varrho, \varrho \mathbf{u})  \Vert^2_{L_2(\mathfrak{U};
W^{-l,2}(\mathcal{K}))}$ is uniformly bounded provided  $\varrho  \in  L^\gamma_{\mathrm{loc}}(\mathbb{R}^3)$ and  $\sqrt{\varrho}\mathbf{u} \in L^2_{\mathrm{loc}}(\mathbb{R}^3)$. See \cite[Eq. 2.3]{Hof}. As such, the stochastic integral $\int_0^\cdot\Phi(\varrho,\varrho \mathbf{u})\mathrm{d}W$ is a well-defined $(\mathscr{F}_t)$-martingale taking value in $W^{-l,2}_{\mathrm{loc}}(\mathbb{R}^3)$.

Lastly, we define the auxiliary space $\mathfrak{U}_0 \supset  \mathfrak{U}$ via
\begin{align*}
\mathfrak{U}_0  =\left\{\mathbf{u}= \sum_{k\geq1}c_ke_k\,;\quad  \sum_{k\geq 1} \frac{c^2_k}{k^2}<\infty \right\}
\end{align*}
and endow it with the norm
\begin{align*}
\Vert  \mathbf{u}  \Vert^2_{\mathfrak{U}_0}  =  \sum_{k\in\mathbb{N}} \frac{c^2_k}{k^2}, \quad   \mathbf{u}=\sum_{k\in\mathbb{N}}c_ke_k.
\end{align*}
Then it can be shown that $W$ has $\mathbb{P}$-a.s. $C([0,T];\mathfrak{U}_0)$ sample paths with the Hilbert--Schmidt embedding $\mathfrak{U}\hookrightarrow \mathfrak{U}_0$. See \cite{da2014stochastic}.

\subsection{Sobolev inequalities for the homogeneous Sobolev space}
\label{sec:homo}
As we shall see shortly, the compactness techniques used in this paper involves certain estimates whose constants must necessarily be independent of the size of the domain. We therefore require the homogeneous Sobolev space 
\[
D^{1,q}(\mathcal{O}) =
\begin{dcases}
\mathbf{u}\in \mathcal{D}'(\mathcal{O}) \,:\, \mathbf{u}\in L^\frac{3q}{3-q}(\mathcal{O}),\, \nabla\mathbf{u} \in L^q(\mathcal{O})  & \text{ if } 1\leq q<3 \\
\mathbf{u}= \{\overline{\mathbf{u}}+c \}_{c\in\mathbb{R}} \,:\, \mathbf{u}\in L^q_{\mathrm{loc}}(\mathcal{O}),\, \nabla\mathbf{u} \in L^q(\mathcal{O})  & \text{ if } q\geq 3
\end{dcases}
\]
which gives such Sobolev-type estimates. Here $\mathcal{O}$ is an exterior or an unbounded domain, for example $\mathcal{O}=\mathbb{R}^3$. In particular, given a function $\mathbf{u}\in D^{1,q}(\mathcal{O})$, we have that for any $1\leq q <3$,
\begin{align}
\Vert \mathbf{u} \Vert_{L^\frac{3q}{3-q}(\mathcal{O})}&\leq c_{q}\Vert\nabla \mathbf{u}\Vert_{L^q(\mathcal{O})} \label{homogeneousPoincare}
\end{align}
See \cite[Chapter II]{galdi2011introduction} for more details. Note that the constant above is independent of  the size of $\mathcal{O}$, unlike in the case of the usual Sobolev--Poinc\'{a}re's inequality.

To continue, let us define the concept of a solution used in this paper.

\begin{definition}
\label{def:martSolution}
If $\Lambda$ is a Borel probability measure on $L^\gamma(\mathbb{R}^3)\times L^\frac{2\gamma}{\gamma+1}(\mathbb{R}^3)$, then we say that
\begin{align}
\label{weakMartSol}
\left[(\Omega,\mathscr{F},(\mathscr{F}_t),\mathbb{P});\varrho, \mathbf{u}, W  \right]
\end{align}
is a \textit{finite energy weak martingale solution} of Eq. \eqref{comprSPDE} with initial law $\Lambda$ provided:
\begin{enumerate}
\item $(\Omega,\mathscr{F},(\mathscr{F}_t),\mathbb{P})$ is a stochastic basis with a complete right-continuous filtration,
\item $W$ is a $(\mathscr{F}_t)$-cylindrical Wiener process,
\item the density $\varrho$ satisfies $\varrho\geq 0,$ $t\rightarrow \langle \varrho(t,\cdot), \phi  \rangle  \,\in \,C[0,T]$ for any $\phi \in C^\infty_c (\mathbb{R}^3)$ $\mathbb{P}-$a.s., the function  $t \mapsto \langle \varrho(t,\cdot), \phi  \rangle$ is progressively measurable, and
\begin{align*}
\mathbb{E}\left[  \sup_{t\in[0,T]}\Vert \varrho(t,\cdot)   \Vert^p_{L^\gamma(\mathcal{K})}  \right]   <   \infty   \text{ for all } 1\leq p   <\infty,
\end{align*}
and for all $\mathcal{K}\subset\mathbb{R}^3$ with  $\mathcal{K}$ compact,
\item the momentum $\varrho \mathbf{u}$ satisfies $t\rightarrow \langle \varrho \mathbf{u}, \phi  \rangle  \,\in \,C[0,T]$ for any $\phi \in C^\infty_c (\mathbb{R}^3)$ $\mathbb{P}-$a.s., the function  $t \mapsto \langle \varrho \mathbf{u}, \phi  \rangle$ is progressively measurable, and $\text{ for all } 1\leq p   <\infty$
\begin{align*}
\mathbb{E}\left[  \sup_{t\in[0,T]}\Vert \sqrt{\varrho} \mathbf{u}   \Vert^p_{L^2(\mathcal{K})}  \right] <\infty, \quad
\mathbb{E}\left[  \sup_{t\in[0,T]}\Vert \varrho \mathbf{u}   \Vert^p_{L^{\frac{2\gamma}{\gamma+1}}(\mathcal{K})}  \right]   <   \infty   ,
\end{align*}
for all $\mathcal{K}\subset\mathbb{R}^3$,  $\mathcal{K}$ compact,
\item the velocity field $\mathbf{u}$ is $(\mathscr{F}_t)$-adapted, $\mathbf{u}\in  L^p\left( \Omega; L^2\left(0,T;W^{1,2}_{\mathrm{loc}}\left(\mathbb{R}^3 \right)  \right)   \right)$ 
and,
\begin{align*}
\mathbb{E}\left[\left( \int_0^T \Vert  \mathbf{u}  \Vert^2_{W^{1,2}(\mathcal{K})} \mathrm{d}t   \right)^p  \right]  <   \infty   \text{ for all } 1\leq p   <\infty,
\end{align*}
for all $\mathcal{K}\subset\mathbb{R}^3$,  $\mathcal{K}$ compact,
\item $\Lambda = \mathbb{P}\circ (\varrho(0), \varrho \mathbf{u}(0))^{-1}$,
\item for all $\psi \in C^\infty_c (\mathbb{R}^3)$ and $\phi \in C^\infty_c (\mathbb{R}^3)$ and all $t\in [0,T]$, it holds $\mathbb{P}-$a.s.
\begin{align*}
\langle \varrho(t), \psi \rangle   &=  \langle \varrho(0), \psi \rangle  +  \int_0^t \langle \varrho \mathbf{u}, \nabla \psi \rangle  \mathrm{d}s,
\\
\langle \varrho \mathbf{u}(t), \phi \rangle   &=  \langle \varrho \mathbf{u}(0), \phi \rangle  +  \int_0^t \langle \varrho \mathbf{u}\otimes \mathbf{u}, \nabla \phi \rangle  \mathrm{d}s - \nu\int_0^t \langle \nabla \mathbf{u}, \nabla \phi \rangle  \mathrm{d}s 
\\
&-(\lambda+ \nu )\int_0^t \langle \mathrm{div}\mathbf{u}, \mathrm{div} \phi \rangle  \mathrm{d}s    +    \frac{1}{\mathrm{Ma}^2}  \int_0^t \langle \varrho^\gamma, \mathrm{div} \phi \rangle  \mathrm{d}s
\\
&+\int_0^t \langle \Phi(\varrho, \varrho \mathbf{u})\mathrm{d}W, \phi \rangle  ,
\end{align*}
\item for any $1\leq p<\infty$, the  energy estimate 
\begin{equation}
\begin{aligned}
\label{energyEx}
\mathbb{E}\, \Bigg[\sup_{t\in[0,T]}\int_{\mathbb{R}^3}\Bigg(\frac{\varrho\vert\mathbf{u}\vert^2}{2}  &+   H(\varrho) \Bigg)(t)\,\mathrm{d}x \Bigg]^p  + \mathbb{E}\,\left[ \int_{Q_T}\mathbb{S}(\nabla\mathbf{u}):\nabla \mathbf{u}\,\mathrm{d}x\mathrm{d}s\right]^p
\\
&\leq c_p \,\left( 1  + \mathbb{E}\,\left[\int_{\mathbb{R}^3}\left(\frac{\vert \mathbf{q}_0 \vert^2}{2\varrho_0} + H(\varrho(0,\cdot)) \right)\mathrm{d}x\right]^p\right),  
\end{aligned}
\end{equation}
holds where $Q_T:=(0,T)\times\mathbb{R}^3$ and where  
\begin{align}
\label{pressurePotential}
H(\varrho ) = \frac{a}{\gamma -1}\left(\varrho^\gamma  -  \gamma \overline{\varrho}^{\gamma-1}(\varrho - \overline{\varrho})  - \overline{\varrho}^\gamma  \right).
\end{align}
is the \textit{pressure potential} for constants $a,\overline{\varrho}>0$.
\item In addition, \eqref{comprSPDE0}$_1$ holds in the renormalized sense. That is, for any $\phi\in \mathcal{D}'(\mathbb{R}^3)$ and $b \in C^0[0,\infty)\cap  C^1(0,\infty)$ such that $\vert b'(t) \vert \leq ct^{-\lambda_0}$, $t\in(0,1]$, $\lambda_0<1$ and $\vert b'(t) \vert  \leq ct^{\lambda_1}$, $t\geq1$ where $c>0$ and $-1< \lambda_1<\infty$, we have that
\begin{align}
\label{renormalizedCont}
\mathrm{d}\langle b(\varrho),\phi \rangle = \langle b(\varrho)\mathbf{u},\nabla\phi \rangle\mathrm{d}t  -  \langle \left(b(\varrho) - b'(\varrho)\varrho\right)\mathrm{div}\mathbf{u},  \phi \rangle\mathrm{d}t.
\end{align} 
\end{enumerate}
\end{definition}

\begin{remark}
The definition above also holds for functions defined on the periodic space $\mathbb{T}^3_L=([-L,L ]\vert_{\{-L,L\}})^3= (\mathbb{R}\,\vert \, 2L\mathbb{Z})^3$ for any $L\geq1$, rather than on the whole space $\mathbb{R}^3$. In that case, it even suffices to consider just smooth test functions which are not necessarily  compactly supported. See for example \cite{breit2015compressible, breit2015incompressible, Hof}.
\end{remark}

\begin{definition}
\label{def:martSolutionIncompre}
If $\Lambda$ is a Borel probability measure on $L^2_{div}(\mathbb{R}^3)$, then
we say that $\left[(\Omega,\mathscr{F},(\mathscr{F}_t),\mathbb{P}), \mathbf{u}, W  \right]$ is a \textit{weak martingale solution} of Eq. \eqref{incomprSPDE} with initial law $\Lambda$ provided:
\begin{enumerate}
\item $(\Omega,\mathscr{F},(\mathscr{F}_t),\mathbb{P})$ is a stochastic basis with a complete right-continuous filtration,
\item $W$ is a $(\mathscr{F}_t)$-cylindrical Wiener process,
\item $\mathbf{u}$ is $(\mathscr{F}_t)$-adapted, $\mathbf{u}\in   C_w\left( [0,T];L^2_{\mathrm{div}}(\mathbb{R}^3) \right) \cap L^2(0,T;W^{1,2}_{\mathrm{div}}(\mathbb{R}^3)  )$ $ \mathbb{P}-a.s. $ and,
\begin{align*}
\mathbb{E}\left[ \sup_{(0,T)} \Vert  \mathbf{u}  \Vert^2_{L^{2}(\mathbb{R}^3)}  \right]^p
+
\mathbb{E}\left[ \left(\int_0^T \Vert  \mathbf{u} \Vert^p_{W^{1,2}(\mathbb{R}^3)} \mathrm{d}t  \right)^p \right]   <   \infty   \text{ for all } 1\leq p   <\infty,
\end{align*}
\item $\Lambda = \mathbb{P}\circ (\mathbf{u}(0))^{-1}$,
\item for all $\phi \in C^\infty_{c,\mathrm{div}} (\mathbb{R}^3)$ and all $t\in [0,T]$, it holds $\mathbb{P}-$a.s.
\begin{align*}
\langle  \mathbf{u}(t), \phi \rangle   &=  \langle  \mathbf{u}(0), \phi \rangle  +  \int_0^t \left[\langle  \mathbf{u}\otimes \mathbf{u}, \nabla \phi \rangle   - \nu \langle \nabla \mathbf{u}, \nabla \phi \rangle \right] \mathrm{d}s 
+\int_0^t \langle \mathcal{P}\Phi(1, \mathbf{u})\mathrm{d}W, \phi \rangle  ,
\end{align*}
\end{enumerate}
\end{definition}

Existence of weak martingale solutions as defined in Definition \ref{def:martSolutionIncompre} has been shown to exist under suitable growth conditions on the noise term. We refer the reader to \cite{mikulevicius2005global}, albeit stated in the Stratonovich sense. A global-in-space existence result stated in the It\^o form appears to be absent from the literatures although it is certainly expected. However, this is a by product of the singular limit problem that we study in this paper. See Theorem \ref{thm:one} below. For bounded domains, see for example, \cite{capinski1994stochastic, flandoli1995martingale}.

\subsection{Prescribed boundary conditions }
Let assume that the right-hand side of the energy inequality \eqref{energyEx} is finite. Then we can deduce from \eqref{pressurePotential} that
\begin{align}
\label{boundaryCondDensity}
\lim_{\vert x\vert\rightarrow\infty}\varrho(x)=\overline{\varrho}
\end{align}
for some $\overline{\varrho}>0$.  This is because if we apply Taylor's expansion around the constant $\overline{\varrho}$ for the function $f(\varrho)=\varrho^\gamma$, we can rewrite \eqref{pressurePotential} as
\begin{align}
\label{pressurePotentialA}
H(\varrho)=\frac{a\gamma z^{\gamma-2}}{2}(\varrho-\overline{\varrho})^2,\quad z\in[\varrho,\overline{\varrho}] \text{ or }  z\in[\overline{\varrho},\varrho]
\end{align}
and so the boundedness of the left-hand side of \eqref{energyEx} means that the difference $\varrho-\overline{\varrho}\in L^p(\Omega;L^\infty(0,T;L^{\mathrm{min}\{2,\gamma\}}(\mathbb{R}^3)))$ when  \eqref{pressurePotentialA} is substituted into  \eqref{energyEx}.

Furthermore, we also have that $\varrho\vert\mathbf{u}\vert^2  \in L^1(\Omega;L^\infty(0,T;L^1(\mathbb{R}^3)))$ and as such,
\begin{align}
\label{boundaryCondPartialMomem}
\lim_{\vert x\vert\rightarrow\infty}\varrho(x)\vert \mathbf{u}(x)\vert^2=0.
\end{align}
By combining \eqref{boundaryCondDensity} and \eqref{boundaryCondPartialMomem} (keeping in mind that $\overline{\varrho}\neq0$), it is reasonable to impose the boundary condition
\begin{align}
\label{boundaryCondVel}
\lim_{\vert x\vert\rightarrow\infty}\mathbf{u}(x)=0.
\end{align}

\subsection{Main results}

We now state the main results of this paper.

\begin{theorem}
\label{thm:dissi}
Let $\gamma>\frac{3}{2}$ and let $\Lambda$ be a probability law on $L^\gamma(\mathbb{R}^3)\times  L^{\frac{2\gamma}{\gamma+1}}(\mathbb{R}^3) $ satisfying
\begin{align*}
\Lambda  \Big\{ (\varrho, \mathbf{q})\in L^\gamma(\mathbb{R}^3) &\times  L^{\frac{2\gamma}{\gamma+1}}(\mathbb{R}^3)  \, : \,  \varrho\geq 0,
\\
& M_1^\mathcal{K}\leq \int_{\mathcal{K}} \varrho\,\mathrm{d}x\leq M^\mathcal{K}_2, \mathbf{q}\vert_{\{\varrho=0\}}=0,\, \left\vert\frac{\varrho-1}{\varepsilon} \right\vert\leq M^\mathcal{K}_2 \Big\}=1,
\\
&\int_{L_x^\gamma\times L_x^{\frac{2\gamma}{\gamma+1}}}  \left\Vert \frac{1}{2}\frac{\vert\mathbf{q}\vert^2}{\varrho}  + H(\varrho)  \right\Vert^p_{L^1_x}   \mathrm{d}\Lambda(\varrho,  \mathbf{q}) \leq c_{p}<\infty,
\end{align*}
for all $0\leq p<\infty$ and any compact set $\mathcal{K}\subset\mathbb{R}^3$ with constants $0<M_1^\mathcal{K}<M_2^\mathcal{K}$ which are independent of $\varepsilon\in(0,1)$. Also assume that \eqref{noiseSupport} and \eqref{stochCoeffBound} holds. Then there exists a finite energy weak martingale solution of \eqref{comprSPDE} in the sense of Definition \ref{def:martSolution}, with initial law $\Lambda$.
\end{theorem}

\begin{remark}
The assumption $\left\vert\frac{\varrho-1}{\varepsilon} \right\vert\leq M^\mathcal{K}_2 $ given in the law above is not restrictive and can actually be dropped. However, it is needed in the proof of Theorem \ref{thm:one} below.
\end{remark}

\begin{theorem}
\label{thm:one}
Let   $\Lambda$ be a given Borel probability measure on $L^2(\mathbb{R}^3)$ and for $\varepsilon\in(0,1)$, we let $\Lambda_\varepsilon$ be a Borel probability measure on $L^\gamma(\mathbb{R}^3)\times  L^{\frac{2\gamma}{\gamma+1}}(\mathbb{R}^3)$ where $\gamma>3/2$ is such that the initial law in Theorem \ref{thm:dissi} holds and where  the marginal law of $\Lambda_\varepsilon$ corresponding to the second component converges to $\Lambda$ weakly in the sense of measures on $L^{\frac{2\gamma}{\gamma+1}}(\mathbb{R}^3)$. If $\left[(\Omega^\varepsilon,\mathscr{F}^\varepsilon,(\mathscr{F}^\varepsilon_t),\mathbb{P}^\varepsilon);\varrho_\varepsilon, \mathbf{u}_\varepsilon, W_\varepsilon  \right]$ is a finite energy weak martingale solution of \eqref{comprSPDE} with initial law $\Lambda_\varepsilon$, then
\begin{align*}
(\varrho_\varepsilon-1) \rightarrow 0\quad&\text{in law in}\quad L^\infty(0,T;L^{\min\{2,\gamma\}}(\mathbb{R}^3)) 
\\
 \mathbf{u}_\varepsilon \rightarrow  \mathbf{u}\quad&\text{in law in}\quad \left(L^2(0,T;W_{\mathrm{loc}}^{1,2}(\mathbb{R}^3)),w\right)
\\
\varrho_\varepsilon  \mathbf{u}_\varepsilon \rightarrow  \mathbf{u}\quad&\text{in law in}\quad L^2(0,T;L_{\mathrm{loc}}^{r}(\mathbb{R}^3)) 
\end{align*}
where $\mathbf{u}$ is a weak martingale solution of \eqref{incomprSPDE} in the sense of Definition \ref{def:martSolutionIncompre} with the initial law $\Lambda$ and $r\in(\frac{3}{2},6)$. 
\end{theorem}

\section{Proof of Theorem \ref{thm:dissi}}
\label{existence}

Let $\varrho_L$ and $\mathbf{u}_L$ be some density and velocity fields defined $\mathrm{d}\mathbb{P}\times \mathrm{d}t$ a.e. $(\omega,t)\in \Omega\times[0,T]$ on the space $\mathbb{T}^3_L$ such that  $\varrho_L$ and $\mathbf{u}_L$ satisfies the so-called \textit{dissipative} estimate; existence of which is shown in \cite[Eq. 3.2]{breit2015compressible} for the particular choice of $L=1$.

We observe  that  \cite[Eq. 3.2]{breit2015compressible} is translation invariant and as such, holds true for any fixed $L\geq1$. Also, the inequality is preserved if we replace $H_\delta(\varrho)$ by $H(\varrho)$. As such if we consider $\psi=\chi_{[0,t]}$, then we obtain the inequality:
\begin{equation}
\begin{aligned}
\label{dissipation1}
&\int_0^t \int_{\mathbb{T}^3_L}\mathbb{S}(\nabla\mathbf{u}_L):\nabla \mathbf{u}_L\,\mathrm{d}x\mathrm{d}s
+ \int_{\mathbb{T}^3_L}\left[\frac{\varrho_L(t)\vert\mathbf{u}_L(t)\vert^2}{2}+  H(\varrho_L(t)) \right]\mathrm{d}x 
\\
&\leq \int_{\mathbb{T}^3_L}\left[\frac{\vert (\varrho_L\mathbf{u}_L)(0)\vert^2}{2\varrho_L(0)} + H(\varrho_L(0)) \right]\mathrm{d}x
+  \int_0^t  \int_{\mathbb{T}^3_L} \mathbf{u}_L\cdot\Phi(\varrho_L,\varrho_L \mathbf{u}_L)\mathrm{d}x\,\mathrm{d}W
\\
&+\int_0^t \int_{\mathbb{T}^3_L}\sum_{k\in\mathbb{N}}\frac{\vert g_k(\varrho_L,\varrho_L\mathbf{u}_L)\vert^2}{2\varrho_L} \,\mathrm{d}x\mathrm{d}s
\end{aligned}
\end{equation}
However, due to \eqref{noiseSupport}, there is a compact set $\mathcal{K}\subset \mathbb{R}^3$ such that for any $1\leq p<\infty$, we have that
\begin{align*}
\mathbb{E}\,\sup_{t\in[0,T]}&\Bigg\vert \int_0^t \int_{\mathbb{T}^3_L} \sum_{k\in\mathbb{N}} \frac{ \vert g_k(\varrho_L,\varrho_L \mathbf{u}_L)\vert^2 }{2\varrho_L}\mathrm{d}x\,\mathrm{d}s   \Bigg\vert^p  
\\
&\leq  \mathbb{E}\, \Bigg(\int_0^T \int_{\mathbb{T}^3_L} \sum_{k\in\mathbb{N}}\frac{ \vert g_k(\varrho_L,\varrho_L \mathbf{u}_L)\vert^2 }{2\varrho_L}\mathrm{d}x\,\mathrm{d}s  \Bigg)^p
\\
&\leq c\, \mathbb{E}\, \Bigg(\int_0^T \int_{\mathcal{K}} \varrho_L^{-1}\left( \varrho_L^2 +\vert\varrho_L\mathbf{u}_L\vert^2  \right) \mathrm{d}x\,\mathrm{d}s  \Bigg)^p
\\
&\leq  
c_p \,\mathbb{E}\,   \int_0^T\Bigg(  \int_{\mathcal{K}} (1  + \varrho_L^\gamma + \varrho_L  \vert \mathbf{u}_L \vert^2)\mathrm{d}x\Bigg)^p\mathrm{d}s
\end{align*}
where $c_p$ is independent of both $k$ and $L$ and where we have used $\varrho_L\leq 1 +\varrho_L^\gamma$.

Also, by the use of the Burkholder--Davis--Gundy inequality, H\"{o}lder inequality and  Young's inequality, we have that
\begin{align*}
\mathbb{E}\,&\Bigg[\sup_{t\in[0,T]}\Big\vert \int_0^t\int_{\mathbb{T}^3_L} \mathbf{u}_L\cdot\Phi(\varrho_L,\mathbf{q}_L)\mathrm{d}x\mathrm{d}W \Big\vert\Bigg]^p  \\
&=  \mathbb{E}\,\Bigg[ \sup_{t\in[0,T]}\Big\vert \int_0^t \sum_{k\in\mathbb{N}} \int_{\mathbb{T}^3_L} \mathbf{u}_L\cdot g_k(\varrho_L,\mathbf{q}_L)\mathrm{d}x\mathrm{d}\beta_k \Big\vert\Bigg]^p
\\
&\leq c_p\,  \mathbb{E}\,\Bigg[\int_0^T   \sum_{k\in\mathbb{N}}\Bigg(\int_{\mathbb{T}^3_L} \mathbf{u}_L \cdot g_k(\varrho_L,\mathbf{q}_L)\mathrm{d}x \Bigg)^2\mathrm{d}s \Bigg]^{\frac{p}{2}}
\\
&\leq c_p\,
\mathbb{E}\,\Bigg[   \int_0^T  \sum_{k\in\mathbb{N}} \Bigg(\int_{\mathbb{T}^3_L}\vert\sqrt{\varrho_L} \mathbf{u}_L\vert^2\mathrm{d}x\Bigg)\Bigg(  \int_{\mathbb{T}^3_L} \Big\vert\frac{g_k(\varrho_L,\mathbf{q}_L)}{\sqrt{\varrho_L}}\Big\vert^2\mathrm{d}x\Bigg)\mathrm{d}s \Bigg]^{\frac{p}{2}}
\\
&\leq \epsilon \,
\mathbb{E}\,   \Bigg(\sup_{t\in[0,T]} \int_{\mathbb{T}^3_L}\vert\sqrt{\varrho_L} \mathbf{u}_L\vert^2\mathrm{d}x\Bigg) ^p + c_{p,\epsilon}\,  \mathbb{E}\,   \int_0^T\Bigg(  \int_{\mathcal{K}} (1  + \varrho_L^\gamma + \varrho_L  \vert \mathbf{u}_L \vert^2)\mathrm{d}x\Bigg)^p\mathrm{d}s 
\end{align*}
for an arbitrarily small $\epsilon>0$.

By taking the $p{\mathrm{th}}$-moment of the supremum  in \eqref{dissipation1} and applying Gronwall's lemma, we obtain the inequality
\begin{equation}
\begin{aligned}
\label{18}
\mathbb{E}\Bigg[ \sup_{t\in[0,T]} &\int_{\mathbb{T}^3_L}\Bigg( \frac{\varrho_L\vert\mathbf{u}_L\vert^2}{2} +  H(\varrho_L) \Bigg)\mathrm{d}x  \Bigg]^p  + \mathbb{E}\Bigg[ \int_0^T \int_{\mathbb{T}^3_L}\mathbb{S}(\nabla\mathbf{u}_L):\nabla \mathbf{u}_L\mathrm{d}x\mathrm{d}s  \Bigg]^p
\\
&\leq c_{p,\epsilon,\mathrm{vol}(\mathcal{K})} \,\left( 1  + \mathbb{E}\,\left[\int_{\mathbb{T}^3_L}\left[\frac{\vert \mathbf{q}_{L,0} \vert^2}{2\varrho_{L,0}} + H(\varrho_L(0,\cdot)) \right]\mathrm{d}x \right]^p\right) 
\end{aligned}
\end{equation}
where $c_{p,\epsilon,\mathrm{vol}(\mathcal{K})}$ is in particular, independent of $L$. Now by the assumptions on $\Lambda$, the right hand side of \eqref{18} is finite.  As such, we obtain the following uniform bounds in $L$
\begin{equation}
\begin{aligned}
\label{est0}
\sqrt{\varrho_L}\mathbf{u}_L &\in L^p\left(\Omega; L^\infty(0,T;L^2(\mathbb{T}^3_L)) \right),
\\
\nabla \mathbf{u}_L &\in L^p\left(\Omega;L^2(0,T;L^{2}(\mathbb{T}^3_L)) \right),
\\
H(\varrho_L) &\in L^p \left(\Omega; L^\infty(0,T;L^{1}(\mathbb{T}^3_L)) \right),
\\
(\varrho_L-\overline{\varrho}) &\in L^p \left(\Omega; L^\infty(0,T;L^{\min\{2,\gamma\}}(\mathbb{T}^3_L)) \right).
\end{aligned}
\end{equation}
Note that the estimates in \eqref{est0} are global but unfortunately, do not include all necessary quantities. In the following, we derive local estimates with respect to balls $B_r$ which will depend on the radius $r>0$. A consequence of \eqref{est0}$_3$ is
\begin{align}\label{eq:74}
\varrho_L &\in L^p \left(\Omega; L^\infty(0,T;L^\gamma(B_r)) \right)
\end{align}
uniformly in $L$ (but depending on $r$). If $B_r\subset \mathbb{T}^3_L$, this follows in an obvious way from the definition of $H$. Otherwise we cover
$B_r\subset\mathbb R^3$ by tori to which $\varrho_L$ is extended by means of periodicity. The number of necessary tori depends on $r$ but is independent of $L$. To see this, we notice that since 
$\mathrm{vol}(B_r)\approx c(\pi)r^3$ and $\mathrm{vol}(\mathbb{T}^3_L)\approx c(\pi)L^3$, we will require $\scriptstyle\mathcal{O} \left(\frac{r^3}{L^3}\right)$ number of tori to cover $B_r$. But since $L\geq 1$,  we infact require $\scriptstyle\mathcal{O}(r^3)$ (which is independent of $L$) number of such tori to cover $B_r$.
\begin{remark}
We get \eqref{eq:74} be making it the subject in \eqref{pressurePotential} and using \eqref{est0}$_{3,4}$. However, we only obtain the estimate locally in space because of the constant term $\overline{\varrho}$ in the pressure potential \eqref{pressurePotential}. This will blow up with the size of the torus if we try obtaining a global estimate.
\end{remark}
We observe that non of the bounds in \eqref{est0} directly controls the amplitude of $\mathbf{u}_L$. 
However using the Sobolev-Poincar\'e's inequality and $\gamma>\frac{3}{2}$, the following holds
\begin{align*}
\| \varrho_0 &\|_{L^1(B_r)}  \left| (\mathbf{u}_L)_{B_r} \right| = \left|\, \int_{B_r} \varrho \,  (\mathbf{u}_L)_{B_r}  \mathrm{d}x \right| \\
&\leq\,c
\int_{B_r} \varrho \left| (\mathbf{u}_L)_{B_r}  - \mathbf{u}_L \right| \mathrm{d}x + \int_{B_r} \varrho_L |\mathbf{u}_L|  \mathrm{d}x\\
 &\leq\,c \| \varrho_L \|_{L^\gamma(B_r)}
\left\| (\mathbf{u}_L)_{B_r}  - \mathbf{u}_L \right\|_{L^{\gamma'}(B_r)} +c \| \sqrt{\varrho_L} \|_{L^2(B_r)} \| \sqrt{\varrho_L} \mathbf{u}_L \|_{L^2(B_r)}\\
 &\leq\,c(r) \| \varrho_L \|_{L^\gamma(B_r)}
\left\| (\mathbf{u}_L)_{B_r}  - \mathbf{u}_L \right\|_{L^{6}(B_r)} +c \| \sqrt{\varrho_L} \|_{L^{2\gamma}(B_r)} \| \sqrt{\varrho_L} \mathbf{u}_L \|_{L^2(B_r)}\\
&\leq c(r)\| \varrho_L \|_{L^\gamma(B_r)} \| \nabla \mathbf{u}_L \|_{L^2(B_r)} +c \| \varrho_L \|^\frac{1}{2}_{L^\gamma(B_r)}  + c\left\| \varrho_L |\mathbf{u}_L|^2 \right\|_{L^1(B_r)},
\end{align*}
and, consequently,
\begin{align}\label{eq:24.04}
\begin{aligned}
\| \varrho_0 \|_{L^1(B_r)}^2  \int_0^\tau \left| (\mathbf{u}_L)_{B_r} \right|^2 \mathrm{d}t &\leq\,c(r)
\sup_{t \in [0,\tau]} \left\| \varrho_L \right\|_{L^\gamma(B_r)}^2 \int_0^\tau \| \nabla \mathbf{u}_L \|^2_{L^2(B_r)}\mathrm{d}t \\ &+c \tau \sup_{t \in (0,\tau)}
\left( \| \varrho_L \|_{L^\gamma(B_r)}  + \left\| \varrho_L |\mathbf{u}_L|^2 \right\|^2_{L^1(B_r)} \right).
\end{aligned}
\end{align}
In view of the bounds established in \eqref{est0}, \eqref{eq:74} and the assumptions on the initial law,
we can conclude that 
\begin{align}
\label{homogVelo}
 \mathbf{u}_L &\in L^p(\Omega;L^2(0,T;W^{1,2}(B_r)) ).
\end{align}
uniformly in $L$.

Furthermore, for $r>0$, we can use  the (uniform in $L$ but not in $r$) continuous embedding $W^{1,2}(B_r)\hookrightarrow L^6(B_r)$ and H\"{o}lder's inequality, to get for $\mathrm{d}\mathbb{P}\times \mathrm{d}t$ a.e. $(\omega,t)\in \Omega\times[0,T]$,
\begin{align*}
\Vert \varrho_L \mathbf{u}_L \Vert_{L^{\frac{2\gamma}{\gamma+1}}(B_r)}  
&\leq
\Vert \sqrt{\varrho_L} \Vert_{L^{2\gamma}(B_r)} \Vert \sqrt{\varrho_L} \mathbf{u}_L \Vert_{L^{2}(B_r)}
\\
&=  \Vert \varrho_L \Vert^\frac{1}{2}_{L^{\gamma}(B_r)}\Vert \sqrt{\varrho_L} \mathbf{u}_L \Vert_{L^{2}(B_r)},
\\
\Vert \varrho_L \mathbf{u}_L\otimes  \mathbf{u}_L \Vert_{L^{\frac{6\gamma}{4\gamma+3}}(B_r)}  
&\leq
 \Vert \varrho_L \mathbf{u}_L \Vert_{L^{\frac{2\gamma}{\gamma+1}}(B_r)}\Vert  \mathbf{u}_L\Vert_{L^{6}(B_r)}.
\end{align*}
Since the radius of the ball above is chosen arbitrarily,  we may conclude that

\begin{equation}
\begin{aligned}
\label{est00}
\varrho_L \mathbf{u}_L &\in L^p(\Omega; L^\infty(0,T;L^{\frac{2\gamma}{\gamma+1}}(B_r)) ),
\\
\varrho_L \mathbf{u}_L\otimes  \mathbf{u}_L &\in L^p(\Omega; L^2(0,T;L^{\frac{6\gamma}{4\gamma+3}}(B_r)) ),
\end{aligned}
\end{equation}
uniformly in $L$ for $r>0$ by using \eqref{est0}.

\subsection{Higher integrability of density}
For reasons that will be clear in the subsequent sections, it is essential to improve the regularity of density. We give this in the following lemma:
\begin{lemma}
\label{lem:higherDensity}
Let $B_r\subset \mathbb{R}^3$ be a ball of radius $r>0$. Then for all $\Theta\leq \frac{2}{3}\gamma-1$, we have that
\begin{align}
\mathbb{E}\int_0^T\int_{B_r} a\varrho_L^{\gamma+\Theta}\,\mathrm{d}x\,\mathrm{d}t\leq c
\end{align}
where the constant $c$, is independent of $L$ {(but depends on $r$)}.
\end{lemma}
\begin{proof} 
If we set $B^3_{r,L}:=B_r\cap \mathbb{T}_L^3$, then it is enough to prove that
\begin{align}
\mathbb{E}\int_0^T\int_{B^3_{r,L}} a\varrho_L^{\gamma+\Theta}\,\mathrm{d}x\,\mathrm{d}t\leq c
\end{align}
independently of $L$. The general case then follows by covering $B_r$ by sets of the form $B\cap \mathbb{T}_L^3$ for a ball $B$.
First notice that by combining \eqref{homogeneousPoincare} with the continuity property of the  Bogovski\u{\i} operator
$\mathbb{B}(\varrho_L^\Theta)=\mathcal{B}\left[\varrho_L^\Theta - \fint \varrho_L^\Theta\,\mathrm{d}x \right]$, where $\mathcal{B}=\mathcal{B}_{B^3_{r,L}}$ is as defined in  \cite[Theorem 5.2]{diening2010decomposition} for the set $B^3_{r,L}$, we ensures that
\begin{align}
\label{BogovEst2}
\Vert \mathbb{B}(\varrho_L^\Theta)\Vert_{L^\frac{3q}{3-q}(B^3_{r,L})}\leq c \Vert \varrho_L^\Theta \Vert_{L^q(B^3_{r,L})}, \quad r>0
\end{align}
holds uniformly in $L$ for $1\leq q<3$.
\begin{remark}
Note that infact the set $B^3_{r,L}$ is a bounded John domain and hence satisfies the emanating chain condition with some constants $\sigma_1$ and $\sigma_2$ which are independent of the size of the torus. The fact that the constant $c$ in \eqref{BogovEst2} is independent of $L$ therefore follows from the fact that the constant $c$ in  \cite[Theorem 5.2]{diening2010decomposition} only depends on $\sigma_1$, $\sigma_2$ and $q$ as well as the fact that $c_{q}$ is independent of $L$.
\end{remark}

The idea now is to test the momentum equation with  $\mathbb{B}(\varrho^\Theta)$. To do this however, we first replace the map $\varrho \,\mapsto\, \varrho^\Theta$ with the function $b(\varrho)\in C_c^1(\mathbb{R})$ and apply It\'{o} formula to the function $f(b ,\mathbf{q}) = \int_{B^3_{r,L}}\mathbf{q}\cdot \mathbb{B}(b(\varrho))\,\mathrm{d}x$ where  $\mathbb{B}(b(\varrho))=\mathcal{B}\left[b(\varrho) - \fint b(\varrho)\,\mathrm{d}x \right]$. Since $f$ is linear in $\mathbf{q}$, no second-order derivative in this component exits. Also, the quadratic variance of $b(\varrho)$ is zero since the renormalized continuity equation  is deterministic.

Now, notice that the  Bogovski\u{\i} operator commutes with the time derivative (but not with the spatial derivative) and since the continuity equation is satisfied in the renormalized sense, we have that
\begin{align*}
\mathrm{d}\left[\mathbb{B}\left(b(\varrho_L)\right)\right]=
\mathbb{B}\left[\mathrm{d}\left(b(\varrho_L)\right)\right]
=
-\mathbb{B} \left[ \mathrm{div}( b(\varrho_L) \mathbf{u}_L)
 -
\left(b'(\varrho_L)\,\varrho_L  -  b(\varrho_L) \right) \mathrm{div}\,\mathbf{u}_L\right] \mathrm{d}t. 
\end{align*}
As such for $b_L:=b(\varrho_L) $, the following holds in expectation:
\begin{align*}
\int_0^t &f_{b_L}(b_L,\mathbf{q}_L) \,\mathrm{d}b_L = \iint\mathbf{q}_L\cdot\partial_{b_L}(\mathbb{B}(b_L))\,\mathrm{d}b_L\,\mathrm{d}x  = \iint\mathbf{q}_L\cdot\,\mathrm{d}\left[\mathbb{B}\left(b_L\right)\right]\,\mathrm{d}x.
\nonumber\\
&= - \iint\mathbf{q}_L\cdot\,  \mathbb{B} \left[ \mathrm{div}( b_L \mathbf{u}_L)\right]\mathrm{d}x\,\mathrm{d}s  
-  \iint\mathbf{q}_L\cdot\, \mathbb{B}\left[\left(\varrho_L\,b'_L  -  b_L \right) \mathrm{div}\,\mathbf{u}_L\right] \mathrm{d}x\,\mathrm{d}s  
\\
\int_0^t  &f_{\mathbf{q_L}}(b_L  ,\mathbf{q}_L) \,\mathrm{d}\mathbf{q}_L  =\iint\mathbb{B}(b_L)\,\mathrm{d}\mathbf{q}_L\,\mathrm{d}x
 \nonumber\\
&=\iint\mathbb{B}(b_L)\,\left[ -\mathrm{div}(\varrho_L\mathbf{u}_L\otimes \mathbf{u}_L) + \nu\Delta\mathbf{u}_L +(\lambda+ \nu)\nabla\mathrm{div}\mathbf{u}_L - a\nabla\varrho^\gamma_L    \right]\,\mathrm{d}x\,\mathrm{d}s 
 \nonumber\\
&+\iint\mathbb{B}(b_L)\Phi(\varrho_L,\varrho_L\mathbf{u}_L)\,\mathrm{d}W\,\mathrm{d}x
\\
&=\iint \Big[ (\varrho_L\mathbf{u}_L\otimes \mathbf{u}_L)\nabla\mathbb{B}(b_L)\mathrm{d}x\mathrm{d}s  - \nu\nabla\mathbf{u}_L:\nabla \mathbb{B}(b_L)
 -(\lambda+ \nu) b_L\,\mathrm{div}\mathbf{u}_L\Big] \mathrm{d}x\mathrm{d}s   
 \nonumber\\
&+ \iint a\varrho_L^\gamma b_L\,\mathrm{d}x\,\mathrm{d}s  +\iint\mathbb{B}(b_L)\Phi(\varrho_L,\varrho_L\mathbf{u}_L)\,\mathrm{d}W\,\mathrm{d}x \label{inverseOgov}
\\
\int_0^t  &f_{b_Lb_L}(b_L,\mathbf{q}_L) \,\mathrm{d}\langle b_L\rangle=\int_0^t  f_{\mathbf{q}_L\mathbf{q}_L}(b_L,\mathbf{q}_L) \,\mathrm{d}\langle\mathbf{q}_L\rangle  =0 \quad\text{since } \mathrm{d}\langle b_L\rangle  =f_{\mathbf{q}_L\mathbf{q}_L}=0  
\end{align*}
where we have integrated by parts and used the fact that $\mathbb{B}(f)$ solves the equation $\mathrm{div}\,\mathbf{v}=f$. 
It therefore follows that 
\begin{equation}
\begin{aligned}
\label{estimates}
&\mathbb{E}  \int_{B^3_{r,L}}\mathbf{q}_L\cdot \mathbb{B}\left(b_L\right)\, \mathrm{d}x  =  \mathbb{E} \int_{B^3_{r,L}}\mathbf{q}_L(0)\cdot \mathbb{B}\left[b_L(0)\right]\, \mathrm{d}x 
\\
&- \mathbb{E} \int_0^t \int_{B^3_{r,L}}  \mathbf{q}_L\cdot\, \mathbb{B} \left[ \mathrm{div}( b_L\, \mathbf{u}_L)\right] \mathrm{d}x\,\mathrm{d}s
\\
&-   \mathbb{E} \int_0^t \int_{B^3_{r,L}}  \mathbf{q}_L\cdot \mathbb{B}\left[\varrho_L b'_L\, \mathrm{div}\mathbf{u}_L\right]  \mathrm{d}x\mathrm{d}s   +  \mathbb{E} \int_0^t \int_{B^3_{r,L}}  \mathbf{q}_L\cdot\,\mathbb{B}\left[b_L  \mathrm{div}\,\mathbf{u}_L\right]\mathrm{d}x\mathrm{d}s
\\
&+ \mathbb{E} \int_0^t \int_{B^3_{r,L}}(\varrho_L\mathbf{u}_L\otimes \mathbf{u}_L)\nabla \mathbb{B}(b_L)\,\mathrm{d}x\mathrm{d}s  -  \mathbb{E} \int_0^t \int_{B^3_{r,L}} \nu\nabla\mathbf{u}_L:\nabla\mathbb{B}(b_L)\,\mathrm{d}x\mathrm{d}s 
 \\
& - \mathbb{E}\, \int_0^t \int_{B^3_{r,L}} (\lambda+ \nu)b_L\, \mathrm{div}\mathbf{u}_L\,\mathrm{d}x\,\mathrm{d}s   +  \mathbb{E}\, \int_0^t \int_{B^3_{r,L}} a\varrho_L^{\gamma}b_L \,\mathrm{d}x\,\mathrm{d}s  
 \\
&+ \mathbb{E}\, \int_0^t \int_{B^3_{r,L}} \mathbb{B}(b_L)\Phi(\varrho_L,\varrho_L\mathbf{u}_L)\,\mathrm{d}W\,\mathrm{d}x
\, =: \,  \mathbb{E}\, \sum_{i=1}^{9}J_i.
\end{aligned}
\end{equation}

To improve the regularity of $\varrho$, we aim at estimating $J_8$ in terms of the rest. To do this, we first set  the left-hand side of \eqref{estimates} to  $\mathbb{E}\,J_0$. Then using  \eqref{homogeneousPoincare}, \eqref{est0}, \eqref{homogVelo}, \eqref{est00} and heavy reliance on H\"{o}lder inequalities, we can show just as in \cite[Propositions 5.1, 6.1]{Hof} for $\delta=0$ and noting that $\Delta^{-1}\nabla$ and $\mathcal{B}$ enjoys the same continuity properties;
\begin{align*}
\mathbb{E}\, J_i\leq c, \quad \text{for all}\quad i\in \{0,1,\ldots, 9\}\setminus \{8\}
\end{align*}
for some constants $c=c_{\Theta,\gamma}$ which are in particular, independent of $L$.

\begin{remark}
In estimating  $J_2$, we use instead, the Bogovski\u{\i} operator in negative spaces which can be found in \cite[Proposition 2.1]{geissert2006equation}, \cite{borchers1990equations} or \cite{farwig1994generalized}. Also, note the comment just after \cite[Remark 2.2]{geissert2006equation} about carrying over the properties of the Bogovski\u{\i} operator from a star shaped domain onto more common domains treated in the analysis of PDE's.
\end{remark}

The result follows by making $\mathbb{E}\,J_8$ the subject and estimating it from above by the estimates given by the rest.
\end{proof}

\subsection{Compactness}
We now show that not only are our earlier estimates bounded uniformly on the torus $\mathbb{T}^3_L$ but due to the fact that each constants obtained are uniform in $L$, they are indeed bounded locally on the whole space $\mathbb{R}^3$. We then proceed to show the usual compactness arguments.

\begin{lemma}
\label{locLemma}
For any $L\geq1$, we have that 
\begin{align*}
\mathbf{u}_L &\in L^p(\Omega;L^2(0,T;W^{1,2}_{\mathrm{loc}}(\mathbb{R}^3)) ),
&
\sqrt{\varrho_L}\mathbf{u}_L &\in L^p\left(\Omega; L^\infty(0,T;L^2_{\mathrm{loc}}(\mathbb{R}^3)) \right),
\\
\varrho_L &\in L^p\left(\Omega;L^\infty(0,T;L^\gamma_{\mathrm{loc}}(\mathbb{R}^3)) \right),
&
\varrho_L\mathbf{u}_L &\in L^p(\Omega;L^\infty(0,T;L^{\frac{2\gamma}{\gamma+1}}_{\mathrm{loc}}(\mathbb{R}^3)) ),
\\
\varrho_L \mathbf{u}_L\otimes  \mathbf{u}_L &\in L^p(\Omega; L^2(0,T;L^{\frac{6\gamma}{4\gamma+3}}_{\mathrm{loc}}(\mathbb{R}^3)) ),
&
\varrho_L &\in L^p(\Omega;L^{\gamma+\Theta}(0,T;L^{\gamma+\Theta}_{\mathrm{loc}}(\mathbb{R}^3)) ).
\end{align*}
uniformly in $L$.
\end{lemma}
\begin{proof}
We will only show the first uniform estimate as the rest can be done in a similar manner in conjunction with \eqref{est0}, \eqref{est00} and Lemma \ref{lem:higherDensity}.

Let $L,r\in\mathbb{N}$ and let $B_r\subset\mathbb{R}^3$ be the ball of radius $r$ centered at the origin. If
$B_r\subset\mathbb{T}^3_L$, then we notice that we can directly deduce from \eqref{est0}$_2$ that
\begin{align}
\label{locVel0}
\mathbf{u}_L \in L^p\left(\Omega;L^2(0,T;W^{1,2}(B_r)) \right)
\end{align}
uniformly in $L$. Otherwise, we can use the same argument as in the justification of \eqref{eq:74} above to get from \eqref{est0}$_2$,
\begin{align}
\label{locVel}
\Vert \mathbf{u}_L\Vert_{L^p(\Omega;L^2(0,T;W^{1,2}(B_r)))}  
\leq c(p,r),\quad\forall r\in\mathbb{N}
\end{align}
uniformly in $L$. That is, for any $r\in\mathbb{N}$ and any $B_r\subset\mathbb{R}^3$, \eqref{locVel} holds. By combining \eqref{locVel0} and \eqref{locVel}, we can deduce that
\begin{align}
\label{locVel1}
\mathbf{u}_L \in L^p\left(\Omega;L^2(0,T;W^{1,2}_{\mathrm{loc}}(\mathbb{R}^3)) \right)
\end{align}
uniformly in $L$.
\end{proof}

For the compactness result, let define the following path space $\chi =  \chi_{\mathbf{u}}\times \chi_\varrho  \times \chi_{\varrho {\mathbf{u}}} \times \chi_W$ where
\begin{align*}
\chi_{\mathbf{u}} &= \left(L^2(0,T;W^{1,2}_{\mathrm{loc}}(\mathbb{R}^3)),\omega\right), \\
\chi_\varrho &= C_\omega\left([0,T];L^\gamma_{\mathrm{loc}}(\mathbb{R}^3)\right)\cap (L^{\gamma+\theta}(0,T;L^{\gamma+\theta}_{\mathrm{loc}}(\mathbb{R}^3)),\omega), \\
\chi_{\varrho {\mathbf{u}}} &= C_\omega\left([0,T];L^{\frac{2\gamma}{\gamma+1}}_{\mathrm{loc}}(\mathbb{R}^3)\right), \\
\chi_W &= C\left([0,T];\mathfrak{U}_0\right) ,
\end{align*}
and let
\begin{enumerate}
\item $\mu_{\mathbf{u}_L}$ be the law of $\mathbf{u}_L$ on $\chi_{{\mathbf{u}}}$,
\item $\mu_{\varrho_L}$ be the law of $\varrho_L$ on the space $\chi_{\varrho}$,
\item $\mu_{\varrho_L \mathbf{u}_L}$ be the law of $\varrho_L \mathbf{u}_L$ on the space $\chi_{\varrho {\mathbf{u}}}$,
\item $\mu_{W}$ be the law of $W$ on the space $\chi_{W}$,
\item $\mathrm{\mu}^L$ be the joint law of $\mathbf{u}_L,\, \varrho_L,\,\varrho_L \mathbf{u}_L$ and $W$ on the space $\chi$.
\end{enumerate}

\begin{proposition}
\label{compact AR}
For an arbitrary constant $c$, which is uniform in $r\in\mathbb{N}$, $L\geq1$ and $R>0$, let us define the set
\begin{align*}
A_R := \{\mathbf{u}_L \in L^2(0,T;W^{1,2}_{\mathrm{loc}}(\mathbb{R}^3) ) \, : \, \Vert \mathbf{u}_L \Vert_{ L^2(0,T;W^{1,2}(B_r) )} \leq {c(r)}R, \quad\forall r\in\mathbb{N}\}.
\end{align*}
Then $A_R$ is compact in $\chi_{\mathbf{u}}$
\end{proposition}
\begin{proof} 
To see this, fix $R>0$ and consider the subsequence $\{ \mathbf{u}_{n}\}_{n\in\mathbb{N}}\subset A_R$ so that
\begin{align*}
\Vert \mathbf{u}_{n} \Vert_{L^2(0,T;W^{1,2}(B_r))}\leq {c(r)}R,\quad \forall n\in\mathbb{N}\,\text{ and }\,\forall r\in\mathbb{N}
\end{align*}
Then by the use of a diagonal argument, we can construct the sequence  $\{\mathbf{u}^n_{n}\}_{n\in\mathbb{N}}\subset  \{\mathbf{u}_{n}\}_{n\in\mathbb{N}}$ that is a common subsequence of all the sequences $\{\mathbf{u}^m_{n}\}_{n\in\mathbb{N}}$ for all $m\in\{0\}\cup\mathbb{N}$ where $\mathbf{u}^0_{n}:=\mathbf{u}_{n}$.  And  by  uniqueness of limits, we can therefore conclude that
\begin{align*}
\mathbf{u}^n_{n} \rightharpoonup \mathbf{u}\quad\text{in}\quad L^2(0,T;W^{1,2}(B_r))\quad\text{for every}\quad r\in\mathbb{N}.
\end{align*}
This finishes the proof.
\end{proof}

\begin{proposition}
\label{tightnessU}
The family of measures $\{\mu^L;\, L\geq 1 \}$  is tight on $\chi$.
\end{proposition}
\begin{proof}
We first show that $\{\mu_{\mathbf{u}_L}; \, L\geq 1 \}$ is tight on $\chi_{\mathbf{u}}$.
To do this, we let $R>0$, then by Proposition \ref{compact AR}, there exists a compact subset $A_R\subset \chi_{\mathbf{u}}$.
Now since 
\begin{align*}
(A_R)^C := \{\mathbf{u}_L \in L^2(0,T;W^{1,2}_{\mathrm{loc}}(\mathbb{R}^3) ) :  \Vert \mathbf{u}_L \Vert_{ L^2\left(0,T;W^{1,2}(B_r) \right)} > {c(r)}R, \,\,\text{for some } r\in\mathbb{N}\},
\end{align*}
for any measure $\mu_{\mathbf{u}_L} \in \{\mu_{\mathbf{u}_L}; \, L\geq 1 \}$, there exists a $r\in\mathbb{N}$ such that:
\begin{align*}
\mu_{\mathbf{u}_L}&\left((A_R)^C\right) = \mathbb{P}\left(\Vert \mathbf{u}_L \Vert_{ L^2\left(0,T;W^{1,2}(B_r) \right)} > {c(r)}R \right)  
\\
&< \frac{1}{{c(r)}R}  \mathbb{E}\left(\Vert \mathbf{u}_L \Vert_{L^2\left(0,T;W^{1,2}(B_r) \right)} \right)  
\leq \frac{1}{R} \rightarrow0.
\end{align*}
as $R\rightarrow\infty$, where we have used \eqref{locVel} in the last inequality. This implies that $\{\mu_{\mathbf{u}_L}; \, L\geq 1 \}$ is tight on $\chi_{\mathbf{u}}$.

By using a similar argument adapted to suit the compactness arguments in \cite[Sect. 6]{Hof} we can show that
$\{\mu_{\varrho_L}; \, L\geq 1\}$ and $\{\mu_{\varrho_L \mathbf{u}_L}; \, L\geq 1 \}$  are also tight on  $\chi_{\varrho}$ and $\chi_{\varrho u}$ respectively.  Furthermore, $\mu_W$ is tight since its a Radon measure on the Polish space $\chi_W$. This finishes the proof.
\end{proof}

From Proposition \ref{tightnessU}, we cannot immediately use Skorokhod representation theorem to deduce that $\{\mu^L\,; \, L\geq1 \}$ is relatively compact (i.e. Prokhorov theorem), since the path space $\chi$ is not metrizable. However, we may use instead the Jakubowski--Skorokhod representation theorem \cite{jakubowski1998short} that gives a similar result but for more general spaces including quasi-Polish spaces, the space in which these locally in space Sobolev functions live. Applying this yields the following result:

\begin{proposition}
\label{prop:Jakubow0}
There exists a subsequence $\mu^n:= \mu^{L_n}$ for $n\in\mathbb{N}$, a probability space $(\tilde{\Omega}, \tilde{\mathscr{F}}, \tilde{\mathbb{P}})$ with $\chi$-valued  random variables $(\tilde{\mathbf{u}}_n,  \tilde{\varrho}_n, \tilde{\mathbf{q}_n}, \tilde{W}_n)$, and their corresponding `limit' variables $( \tilde{\mathbf{u}},  \tilde{\varrho}, \tilde{\mathbf{q}},  \tilde{W})$ such that
\begin{itemize}
\item the law of $(\tilde{\mathbf{u}}_n, \tilde{\varrho}_n,  \tilde{\mathbf{q}}_n, \tilde{W}_n)$ is given by $\mu^n  =  \mathrm{Law}(\mathbf{u}_{L_n}, \varrho_{L_n},  \varrho_{L_n}\mathbf{u}_{L_n},  W)$, $n\in\mathbb{N}$,
\item the law of  $(\tilde{\mathbf{u}}, \tilde{\varrho},  \tilde{\mathbf{q}}, \tilde{W})$, denoted by $\mu =  \mathrm{Law}(\mathbf{u}, \varrho,  \varrho \mathbf{u}, W)$ is a Randon measure,
\item $(\tilde{\mathbf{u}}_n,  \tilde{\varrho}_n,  \tilde{\mathbf{q}}_n,  \tilde{W}_n)$ converges $\tilde{\mathbb{P}}-$a.s to $(\tilde{\mathbf{u}}, \tilde{\varrho},  \tilde{\mathbf{q}},  \tilde{W})$ in the topology of $\chi$.
\end{itemize}
\end{proposition}

To extend this new probability space $(\tilde{\Omega}, \tilde{\mathscr{F}}, \tilde{\mathbb{P}})$ into a stochastic basis, we endow it with a filtration. To do this, let us first define a restriction operator $\textbf{r}_t$ define by
\begin{align}
\label{continuousFunction}
\textbf{r}_t:X\rightarrow X\vert_{[0,t]}, \quad f\mapsto f\vert_{[0,t]},
\end{align}
for $t\in[0,T]$ and $X\in  \{ \chi_\varrho, \chi_{\mathbf{u}},  \chi_W \}$. We observe that $\textbf{r}_t$ is a continuous map. We can therefore construct $\tilde{\mathbb{P}}-$augmented canonical filtrations for $(\tilde{\varrho}_n, \tilde{\mathbf{u}}_n, \tilde{W}_n)$ and  $(\tilde{\varrho}, \tilde{\mathbf{u}},  \tilde{W})$ respectively, by setting
\begin{align*}
\tilde{\mathscr{F}}^n_t = \sigma\left( \sigma(\textbf{r}_t\tilde{\varrho}_n,  \textbf{r}_t\tilde{\mathbf{u}}_n,   \textbf{r}_t\tilde{W}_n)  \cup  \{N\in\tilde{\mathscr{F}};\,\tilde{\mathbb{P}}(N)=0  \} \right), \quad
t\in[0,T],
\\
\tilde{\mathscr{F}}_t = \sigma\left( \sigma(\textbf{r}_t\tilde{\varrho},   \textbf{r}_t\tilde{\mathbf{u}},  \textbf{r}_t\tilde{W})  \cup  \{N\in\tilde{\mathscr{F}};\,\tilde{\mathbb{P}}(N)=0  \} \right), \quad
t\in[0,T].
\end{align*}

The following result thus follows:

\begin{lemma}
\label{renorSeq}
For any $n>0$, $[ (\tilde{\Omega}, \tilde{\mathscr{F}}, (\tilde{\mathscr{F}}^n_t)_{t\geq0}, \tilde{\mathbb{P}}),\tilde{\varrho}_n, \tilde{\mathbf{u}}_n,   \tilde{W}_n]$ is a weak martingale solution of \eqref{comprSPDE} with initial law $\Lambda$. Furthermore,
there exists $b>\frac{3}{2}$ and a $W^{-b,2}(\mathbb{R}^3)-$valued continuous square integrable $(\tilde{\mathscr{F}}_t)-$martingale $\tilde{M}$ and  $\tilde{p} \in L^\frac{\gamma+\Theta}{\gamma}(\tilde{\Omega}\times Q)$, where $Q=(0,T)\times\mathbb{R}^3$, such that  $[ (\tilde{\Omega}, \tilde{\mathscr{F}}, (\tilde{\mathscr{F}}_t)_{t\geq0}, \tilde{\mathbb{P}}),\tilde{\varrho}, \tilde{\mathbf{u}}, \tilde{p},   \tilde{M}]$  is a  weak martingale solution of 
\begin{equation}
\begin{aligned}
\label{0comprSPDE}
\mathrm{d}\tilde{\varrho} + \mathrm{div}(\tilde{\varrho}\tilde{\mathbf{u}})\mathrm{d}t &= 0 \\
\mathrm{d}(\tilde{\varrho}\tilde{\mathbf{u}}) + [\mathrm{div}(\tilde{\varrho}\tilde{\mathbf{u}} \otimes \tilde{\mathbf{u}})-\nu\Delta \tilde{\mathbf{u}} -(\lambda +\nu)\nabla\mathrm{div}\tilde{\mathbf{u}} + \nabla \tilde{p}]\mathrm{d}t &= \mathrm{d}\tilde{M},\quad\text{in }\tilde{\Omega}\times Q
\end{aligned}
\end{equation}
with initial law $\Lambda$. Furthermore, \eqref{0comprSPDE}$_1$ is satisfied in the renormalized sense.
\end{lemma}

\begin{proof}
This follows in exactly the same manner as in \cite[Proposition 5.6]{Hof}.
\end{proof}

\begin{corollary}
\label{cor:almostSure}
The following $\tilde{\mathbb{P}}-$a.s. convergence holds:
\begin{equation}
\begin{aligned}
\label{almostSure}
\tilde{\mathbf{u}}_n \rightharpoonup \tilde{\mathbf{u}}\quad &\text{ in }\quad L^2(0,T;W^{1,2}_{\mathrm{loc}}(\mathbb{R}^3)),
\\
\tilde{\varrho}_n \rightarrow \tilde{\varrho}\quad &\text{ in }\quad  C_\omega([0,T];L^{\gamma}_{\mathrm{loc}}(\mathbb{R}^3)),
\\
\tilde{\varrho}_n \rightharpoonup \tilde{\varrho}\quad &\text{ in }\quad  L^{\gamma+\Theta}(0,T;L^{\gamma+\Theta}_{\mathrm{loc}}(\mathbb{R}^3)),
\\
\tilde{\varrho}_n\tilde{\mathbf{u}}_n \rightarrow \tilde{\varrho}\tilde{\mathbf{u}} \quad &\text{ in } \quad C_\omega([0,T] ;L^\frac{2\gamma}{\gamma+1}_{\mathrm{loc}}(\mathbb{R}^3)) \cap  L^2(0,T ;W^{-1,2}_{\mathrm{loc}}(\mathbb{R}^3)),
\\
\tilde{\varrho}_n \tilde{\mathbf{u}}_n\otimes \tilde{\mathbf{u}}_n  \rightharpoonup \tilde{\varrho}\tilde{\mathbf{u}}\otimes \tilde{\mathbf{u}} \quad  &\text{ in }  \quad  L^1(0,T;L^{1}_{\mathrm{loc}}(\mathbb{R}^3)),
\\
\tilde{W}_n \rightarrow \tilde{W}\quad &\text{ in }\quad  C\left([0,T];\mathfrak{U}_0\right),
\end{aligned}
\end{equation}
\end{corollary}
\begin{proof}
The first three and the last is exactly contained in Proposition \ref{prop:Jakubow0}. For \eqref{almostSure}$_{4,5}$, see \cite[Lemma 5.5, Proposition 6.3]{Hof}. 
\end{proof}

\begin{proposition}
\label{globalEst}
The limit process $\tilde{\mathbf{u}}$ in \eqref{almostSure}
is globally defined in space, i.e., $\tilde{\mathbf{u}}\in L^2(0,T;W^{1,2}(\mathbb{R}^3))$.
\end{proposition}
\begin{proof}
%
%
Let $B_r\subset\mathbb{R}^3$ be an arbitrary ball of radius $r>0$.
Then from \eqref{almostSure}$_1$, we have  that for $\tilde{\mathbb{P}}-$ a.s.,
\begin{align*}
 \tilde{\mathbf{u}}_n \rightharpoonup\tilde{\mathbf{u}}\quad \text{in} \quad L^2(0,T;W^{1,2}(B_r)), \text{ for }r>0.
\end{align*}
However, lower semicontinuity of norms means that for any such $r>0$,
\begin{align*}
\Vert\chi_{B_r} \nabla \tilde{\mathbf{u}} \Vert_{L^2(0,T;L^{2}(\mathbb{R}^3))}= \Vert \nabla\tilde{\mathbf{u}}\Vert_{L^2(0,T;L^{2}(B_r))}\leq \liminf_{n\rightarrow\infty} \Vert \nabla \tilde{\mathbf{u}}_n\Vert_{L^2(0,T;L^{2}(B_r))} 
\end{align*}
$\tilde{\mathbb{P}}-$a.s. Passing to the limit $r\rightarrow\infty$ on either side of this inequality  finishes the proof since by the Gagliardo--Nirenberg--Sobolev inequality, \eqref{homogeneousPoincare} then follows for $q=2$.
\end{proof}


\subsection{The effective viscous flux}
This section combines ideas from \cite{feireisl2001existence, Hof} and \cite[Chapter 7]{straskraba2004introduction}.  

Let $ \Delta^{-1}$ be the inverse Laplacian on $\mathbb{R}^3$ and let the global-in-space operators $\mathcal{A}_i= \Delta^{-1}[\partial_{x_i}u],\quad i=1,2,3$ be as defined in \cite[Sect. 4.4.1]{straskraba2004introduction} or \cite[Sect. 3.4]{feireisl2001existence}.

Then by using the convention $\partial_i:=\partial_{x_i}$ and for some cutoff functions $\phi(x), \underline{\phi}(x)\in C^\infty_c(\mathbb{R}^3)$, we may do a similar computation as in \eqref{estimates}. That is, we apply It\^{o}'s formula to the function
$f(g ,\tilde{\mathbf{q}})=\int_{\mathbb{R}^3} \tilde{\mathbf{q}}\cdot \phi(x)\mathcal{A}_i [\underline{\phi}(x)g ]\,\mathrm{d}x$ where $\tilde{\mathbf{q}}=\tilde{\varrho}\tilde{\mathbf{u}}$ and where $g=T_k(\tilde{\varrho})$ and $T_k:[0,\infty)\rightarrow[0,\infty)$ is given by
\[
\label{Tk}
T_k(t) =
\begin{dcases}
 t  & \text{ if } 0\leq t< k, \\
k & \text{ if } k\leq t<\infty.
\end{dcases}
\]
Or equivalently, by testing the momentum equation satisfied by the sequence of weak martingale solution in Lemma \ref{renorSeq} by $\varphi_i(x)=\phi(x)\mathcal{A}_i[\underline{\phi}(x)T_k(\tilde{\varrho})]$. We obtain the following (by  assuming that $L$ is large enough such that $spt{\phi}\subset \mathbb{T}_L^3$)

\begin{equation}
\begin{aligned}
\label{rieszTrans0}
\tilde{\mathbb{E}} &\, \int_{\mathbb{R}^3}\phi \,\tilde{\varrho}_n\tilde{u}^i_n\, \mathcal{A}_i\left[\underline{\phi}T_k(\tilde{\varrho}_n)\right]\, \mathrm{d}x  =  \tilde{\mathbb{E}}\, \int_{\mathbb{R}^3}\phi\,\tilde{\varrho}_n\tilde{u}^i_n(0)\, \mathcal{A}_i\left[\underline{\phi}T_k(\tilde{\varrho}_n(0))\right]\, \mathrm{d}x
\\
&  - \tilde{\mathbb{E}}\, \int_0^t \int_{\mathbb{R}^3} \phi\,\tilde{\varrho}_n\tilde{u}^i_n\, \mathcal{A}_i[\underline{\phi}\, \partial_j( T_k(\tilde{\varrho}_n) \tilde{u}^j_n)] \mathrm{d}x\,\mathrm{d}s 
\\
& -   \tilde{\mathbb{E}}\, \int_0^t \int_{\mathbb{R}^3}  \phi\,\tilde{\varrho}_n\tilde{u}^i_n\, \mathcal{A}_i\left[\underline{\phi}\left(T_k'(\tilde{\varrho}_n)\,\tilde{\varrho}_n -T_k(\tilde{\varrho}_n) \right) \mathrm{div}\,\tilde{\mathbf{u}}_n\right]  \mathrm{d}x\,\mathrm{d}s   
\\
&+ \tilde{\mathbb{E}}\, \int_0^t \int_{\mathbb{R}^3}\tilde{\varrho}_n\tilde{u}^i_n \tilde{u}^j_n\, \partial_j(\phi\,\mathcal{A}_{i} [\underline{\phi}T_k(\tilde{\varrho}_n)])\,\mathrm{d}x\,\mathrm{d}s  
\\
&+ \nu \tilde{\mathbb{E}}\, \int_0^t \int_{\mathbb{R}^3} \phi\, \mathcal{A}_i[\underline{\phi}T_k(\tilde{\varrho}_n)]\,\Delta\tilde{u}^i_n \,\mathrm{d}x\,\mathrm{d}s 
 \\
&+ \tilde{\mathbb{E}}\, \int_0^t \int_{\mathbb{R}^3}[a\tilde{\varrho}^{\gamma}_n -(\lambda+ \nu) \mathrm{div}\tilde{\mathbf{u}}_n]\,\partial_i( \phi\mathcal{A}_i[\underline{\phi}\,T_k(\tilde{\varrho}_n)])\,\mathrm{d}x\,\mathrm{d}s    
\\
&=: \tilde{\mathbb{E}}\,\sum_{k=1}^{6 }J_k, \quad i=1,2,3.
\end{aligned}
\end{equation}
where $T_k$, as defined above, replaces $b$ in the definition of the renormalized equation given by \eqref{renormalizedCont}.
\begin{remark}
Notice that since  the approximate quantities in \eqref{almostSure} are only defined locally in space, to apply this globally defined operators $\mathcal{A}$, it is essentially to pre-multiply our functions by some $\underline{\phi}\in C^\infty_c(\mathbb{R})$.

Also, we observe that since our noise term is a martingale, it vanishes when we take its expectation, as martingales are constant on average.
\end{remark}
Now notice that by integration  by parts and the use of the properties of the operators $\mathcal{A}_i$ and $\mathcal{R}_{ij}=\partial_i \mathcal{A}_j$, we may rewrite $J_2, J_4, J_5$ and $J_6$ so that \eqref{rieszTrans0} becomes:
\begin{equation}
\begin{aligned}
\label{rieszTrans1}
&\tilde{\mathbb{E}}\, \int_0^t \int_{\mathbb{R}^3}[a\tilde{\varrho}^{\gamma}_n -(\lambda+ 2\nu) \mathrm{div}\tilde{\mathbf{u}}_n]\, \phi\,\underline{\phi}\,T_k(\tilde{\varrho}_n)\,\mathrm{d}x\,\mathrm{d}s
\\
&=
\tilde{\mathbb{E}} \, \int_{\mathbb{R}^3}\phi\,\tilde{\varrho}_n\tilde{u}^i_n\, \mathcal{A}_i\left[\underline{\phi}\,T_k(\tilde{\varrho}_n)\right]\, \mathrm{d}x
\\
 &-\tilde{\mathbb{E}}\, \int_{\mathbb{R}^3}\phi\,\tilde{\varrho}_n\tilde{u}^i_n(0)\, \mathcal{A}_i\left[\underline{\phi}\,T_k(\tilde{\varrho}_n(0))\right]\, \mathrm{d}x
 +\nu \tilde{\mathbb{E}}\, \int_0^t \int_{\mathbb{R}^3}\underline{\phi}\, \tilde{u}^i_n \, T_k(\tilde{\varrho}_n)\,\partial_i\phi \,\mathrm{d}x\,\mathrm{d}s 
 \\
&-\tilde{\mathbb{E}}\, \int_0^t \int_{\mathbb{R}^3}[a\tilde{\varrho}^{\gamma}_n -(\lambda+ \nu) \mathrm{div}\tilde{\mathbf{u}}_n] \mathcal{A}_i[\underline{\phi}\,T_k(\tilde{\varrho}_n)]\,\partial_i\phi \,\mathrm{d}x\,\mathrm{d}s   
\\
& +  \tilde{\mathbb{E}}\, \int_0^t \int_{\mathbb{R}^3}  \phi\,\tilde{\varrho}_n\tilde{u}^i_n\, \mathcal{A}_i\left[\underline{\phi}\,\left(T_k'(\tilde{\varrho}_n)\,\tilde{\varrho}_n -T_k(\tilde{\varrho}_n) \right) \mathrm{div}\,\tilde{\mathbf{u}}_n\right]  \mathrm{d}x\,\mathrm{d}s   
\\
&  + \tilde{\mathbb{E}}\, \int_0^t \int_{\mathbb{R}^3} \tilde{u}^i_n \left(\mathcal{R}_{ij}  [\phi\,\tilde{\varrho}_n\tilde{u}^j_n ]\,\underline{\phi}\, T_k(\tilde{\varrho}_n)  
- \phi\,\tilde{\varrho}_n \tilde{u}^j_n\,\mathcal{R}_{ij} [\underline{\phi}\,T_k(\tilde{\varrho}_n)]\right)\,\mathrm{d}x\,\mathrm{d}s  
\\
& + \tilde{\mathbb{E}}\, \int_0^t \int_{\mathbb{R}^3} \tilde{u}^j_n \left(\mathcal{A}_{i}  [\phi\,\tilde{\varrho}_n\tilde{u}^i_n ]\, T_k(\tilde{\varrho}_n)  \partial_j\underline{\phi}\,
- \tilde{\varrho}_n \tilde{u}^i_n\,\mathcal{A}_{i} [\underline{\phi}T_k(\tilde{\varrho}_n)]\partial_j \phi\,\right)\,\mathrm{d}x\,\mathrm{d}s  
\\
&=:
\tilde{\mathbb{E}}\,\sum_{k=1}^{7}I_k, \quad i=1,2,3.
\end{aligned}
\end{equation}
\begin{remark}
If we set the left-hand side of \eqref{rieszTrans1} to $\tilde{\mathbb{E}}I_0$, then we point the reader to the difference in the viscosity constant in $I_0$ and $I_4$. 
\end{remark}

Similarly for the limit processes, we obtain
\begin{equation}
\begin{aligned}
\label{rieszTrans3}
&\tilde{\mathbb{E}}\, \int_0^t \int_{\mathbb{R}^3}[a\tilde{p} -(\lambda+ 2\nu) \mathrm{div}\tilde{\mathbf{u}}]\, \phi\, \underline{\phi}\,\overline{T_k(\tilde{\varrho})}\,\mathrm{d}x\,\mathrm{d}s
=
\tilde{\mathbb{E}} \, \int_{\mathbb{R}^3}\phi\,\tilde{\varrho}\tilde{u}^i\, \mathcal{A}_i\left[\underline{\phi}\,\overline{T_k(\tilde{\varrho})}\right]\, \mathrm{d}x
\\
 &-\tilde{\mathbb{E}}\, \int_{\mathbb{R}^3}\phi\,\tilde{\varrho}\tilde{u}^i(0)\, \mathcal{A}_i\left[\underline{\phi}\,\overline{T_k(\tilde{\varrho}(0))}\right]\, \mathrm{d}x
 +\nu \tilde{\mathbb{E}}\, \int_0^t \int_{\mathbb{R}^3}\underline{\phi}\, \tilde{u}^i \, \overline{T_k(\tilde{\varrho})}\,\partial_i\phi\, \mathrm{d}x\,\mathrm{d}s 
  \\
&-\tilde{\mathbb{E}}\, \int_0^t \int_{\mathbb{R}^3}[a\overline{\tilde{p}} -(\lambda+ \nu) \mathrm{div}\tilde{\mathbf{u}}] \mathcal{A}_i[\underline{\phi}\,\overline{T_k(\tilde{\varrho})}]\,\partial_i\phi \,\mathrm{d}x\,\mathrm{d}s  
\\
& +  \tilde{\mathbb{E}}\, \int_0^t \int_{\mathbb{R}^3}  \phi\,\tilde{\varrho}\tilde{u}^i\, \mathcal{A}_i\left[\underline{\phi}\,\overline{\left(T_k'(\tilde{\varrho})\,\tilde{\varrho} -T_k(\tilde{\varrho}) \right) \mathrm{div}\,\tilde{\mathbf{u}}}\right]  \mathrm{d}x\,\mathrm{d}s   
\\
&  + \tilde{\mathbb{E}}\, \int_0^t \int_{\mathbb{R}^3} \tilde{u}^i \left(\mathcal{R}_{ij}  [\phi\,\tilde{\varrho}\tilde{u}^j ]\,\underline{\phi}\, \overline{T_k(\tilde{\varrho})}  
- \phi\,\tilde{\varrho} \tilde{u}^j \,\mathcal{R}_{ij} [\underline{\phi}\,\overline{T_k(\tilde{\varrho})}]\right)\,\mathrm{d}x\,\mathrm{d}s  
\\
& + \tilde{\mathbb{E}}\, \int_0^t \int_{\mathbb{R}^3} \tilde{u}^j \left(\mathcal{A}_{i}  [\phi\,\tilde{\varrho}\tilde{u}^i]\, \overline{T_k(\tilde{\varrho})}  \partial_j\underline{\phi}\,
- \tilde{\varrho} \tilde{u}^i \,\mathcal{A}_{i} [\underline{\phi}\overline{T_k(\tilde{\varrho})}]\partial_j \phi\,\right)\,\mathrm{d}x\,\mathrm{d}s  
\\
&=:
\tilde{\mathbb{E}}\,\sum_{k=1}^{7}K_k, \quad i=1,2,3.
\end{aligned}
\end{equation}
where a `bar' above a function represents the limit of the corresponding approximate sequence of functions.

\begin{lemma}
\label{rieszConv}
Let $\phi(x), \underline{\phi}(x) \in C^\infty_c(\mathbb{R}^3)$. Then the strong convergence
\begin{align*}
\mathcal{R} [\phi\,\tilde{\varrho}_n\tilde{u}^j_n ]\, &\underline{\phi}T_k(\tilde{\varrho}_n)  
- \phi\,\tilde{\varrho}_n \tilde{u}^j_n \,\mathcal{R} [\underline{\phi}T_k(\tilde{\varrho}_n)]
&\rightarrow \mathcal{R}  [\phi\,\tilde{\varrho}\tilde{u}^j ]\, \underline{\phi}\overline{T_k(\tilde{\varrho})}  
- \phi\,\tilde{\varrho} \tilde{u}^j \,\mathcal{R} [\underline{\phi}\overline{T_k(\tilde{\varrho})}] 
\end{align*}
holds  in $L^2\left(\tilde{\Omega}\times (0,T); 
{W^{-1,2}(\mathbb{R}^3)}\right)$ where $\mathcal{R}:=\mathcal{R}_{ij}$. 
\end{lemma}
\begin{proof}
See \cite[Sect. 6.1]{Hof} or the deterministic counterpart in \cite[Eq. 7.5.23]{straskraba2004introduction}.
\end{proof}

Now by using the  weak-strong pair: $\eqref{almostSure}_{1}$ and Lemma \ref{rieszConv}, we can pass to the limit in the crucial term $I_6$ to  get $\tilde{\mathbb{E}}\, I_6 \rightarrow \tilde{\mathbb{E}}\, K_6$. 

All other terms can be treated in a similar manner as in \cite[Sect. 6.1]{Hof} keeping in mind that the terms involving derivatives and cutoff functions are of lower order and hence easier to handle. In particular, we obtain the convergence  $\tilde{\mathbb{E}}\, I_7 \rightarrow \tilde{\mathbb{E}}\, K_7$ by observing that $\mathcal{R}=\partial_j\mathcal{A}_i$.

We have therefore shown that
\begin{equation}
\begin{aligned}
\label{effectiveVIsco}
\lim_{n\rightarrow 0}\tilde{\mathbb{E}}\int_Q & \left[a\tilde{\varrho}^\gamma_n-(\lambda+2\nu)\mathrm{div}\,\tilde{\mathbf{u}}_n  \right]\phi\underline{\phi} T_k(\tilde{\varrho}_n) \,\mathrm{d}x\, \mathrm{d}t
\\
&=
\tilde{\mathbb{E}}\int_Q\left[a\tilde{p}-(\lambda+2\nu)\mathrm{div}\,\tilde{\mathbf{u}}  \right]\phi\underline{\phi} \overline{T_k(\tilde{\varrho})}  \,\mathrm{d}x\, \mathrm{d}t
\end{aligned}
\end{equation}
\subsection{Identification of the pressure limit}

Showing that indeed $\tilde{p}=\tilde{\varrho}^\gamma$ or equivalently that $\tilde{\varrho}_n\rightarrow \tilde{\varrho}$ strongly in $L^p(\tilde{\Omega}\times Q)$ for all $p\in[1,\gamma+\Theta)$ follows Feireisl's approach via the use of the so-called \textit{oscillation defect measure}. This is a purely deterministic argument even in our stochastic settings since it relies on the renormalized continuity equation. To avoid repetition, we refer the reader to \cite[Sect. 7.3.7.3]{straskraba2004introduction} or \cite{feireisl2001compactness}. To confirm that it indeed applies in the stochastic setting, the reader may also refer to \cite[Sect. 6.2 and 6.3]{Hof}.

We now conclude with the following lemma which completes the proof of Theorem \ref{thm:dissi}.

\begin{lemma}
$[ (\tilde{\Omega}, \tilde{\mathscr{F}}, (\tilde{\mathscr{F}}_t)_{t\geq0}, \tilde{\mathbb{P}}),\tilde{\varrho}, \tilde{\mathbf{u}},  \tilde{W}]$ is a finite energy weak martingale solution of \eqref{comprSPDE} with initial law $\Lambda$. Furthermore, \eqref{comprSPDE}$_1$ is satisfied in the renormalized sense.
\end{lemma}

\section{Proof of Theorem \ref{thm:one}}
\label{singularLimit}

For every $\varepsilon>0$, let assume there exits a finite energy weak martingale solution of Eq. \eqref{comprSPDE} given by
\begin{align*}
\left[(\Omega^\varepsilon,\mathscr{F}^\varepsilon,(\mathscr{F}^\varepsilon_t),\mathbb{P}^\varepsilon), \varrho_\varepsilon, \mathbf{u}_\varepsilon , W_\varepsilon  \right].
\end{align*}
Then by setting $\overline{\varrho}=1$ and $a=\frac{1}{\varepsilon^2}$ in \eqref{pressurePotential}, and applying Taylor expansion to the function $f(\varrho)=\varrho^\gamma$ around $\varrho=1$, we get
\begin{align*}
\mathbb{E}&\left[\int_{\mathbb{R}^3}\left(\frac{\vert \mathbf{q}_\varepsilon(0)\vert^2}{2\varrho_\varepsilon(0)} +{H(\varrho_\varepsilon(0))}
\right)\mathrm{d}x\right]^p    
\\
&=  \int_{L^\gamma\times L^{\frac{2\gamma}{\gamma+1}}}\left\Vert  \frac{\vert \mathbf{q}_\varepsilon\vert^2}{2\varrho_\varepsilon}  
+ \frac{\gamma z^{\gamma-2}}{2\varepsilon^2}(\varrho_\varepsilon-1)^2 \right\Vert^p_{L^1(\mathbb{R}^3)}\mathrm{d}\Lambda(\varrho_\varepsilon, \mathbf{q}_\varepsilon)
\leq c_{p,T}
\end{align*}
for $z\in[\varrho,1] \text{ or }  z\in[1,{\varrho}]$ and where we have used the initial law in Theorem \ref{thm:dissi}. Similar to Section \ref{existence}, we  can now collect the following uniform (in $\varepsilon$) bounds
\begin{equation}
\begin{aligned}
\label{summary}
\varphi_\varepsilon &\in L^p(\Omega;L^\infty(0,T;L^{\min\{2,\gamma\}}(\mathbb{R}^3))),
\\
 \nabla \mathbf{u}_\varepsilon &\in L^p\left(\Omega;L^2(0,T;L^{2}(\mathbb{R}^3))\right),
 \\
\sqrt{\varrho_\varepsilon} \mathbf{u}_\varepsilon &\in L^p\left(\Omega;L^\infty(0,T;L^{2}(\mathbb{R}^3))\right),
\\
\varrho_\varepsilon \mathbf{u}_\varepsilon  &\in L^p(\Omega; L^\infty(0,T;L_{\mathrm{loc}}^{\frac{2\gamma}{\gamma+1}}(\mathbb{R}^3))),
\\
\varrho_\varepsilon  \mathbf{u}_\varepsilon \otimes \mathbf{u}_\varepsilon &\in L^p(\Omega;L^2(0,T;L_{\mathrm{loc}}^{\frac{6\gamma}{4\gamma +3}}(\mathbb{R}^3)) ),
\end{aligned}
\end{equation}
where $\varphi_\varepsilon:=\frac{\varrho_\varepsilon-1}{\varepsilon}$ and where
 \begin{align}
 \label{convergenceofDensity}
 \varrho_\varepsilon \rightarrow 1\quad \text{in }L^p\left(\Omega;L^\infty(0,T;L^\gamma_{\mathrm{loc}}(\mathbb{R}^3)) \right).
 \end{align}
 cf. \eqref{est0} (with $\overline{\varrho}=1$), \eqref{est00} and \cite[eqn. 3.6]{breit2015incompressible}.


\subsection{Acoustic wave equation}
Let $\Delta^{-1}$ represent the inverse of the Laplace operator on $\mathbb{R}^3$ and let $\mathcal{Q}= \nabla\Delta^{-1}\mathrm{div}$ and $\mathcal{P}$ be, respectively, the gradient and solenoidal parts according to Helmoltz decomposition. Then referring again to \cite{breit2015incompressible}, we observe that by setting $\varphi_\varepsilon= \frac{\varrho_\varepsilon-1}{\varepsilon}$ and $\mathrm{Id} = \mathcal{Q}+ \mathcal{P}$, we derive from equation \eqref{comprSPDE}:
\begin{equation}
\begin{aligned}
\label{acousticSPD}
\varepsilon\mathrm{d}\varphi_\varepsilon +  \mathrm{div}\mathcal{Q}(\varrho_\varepsilon \mathbf{u}_\varepsilon)\mathrm{d}t   &=0,
\\
\varepsilon  \mathcal{Q}\Phi(\varrho_\varepsilon,\varrho_\varepsilon \mathbf{u}_\varepsilon)\mathrm{d}W  - \gamma \nabla \varphi_\varepsilon\mathrm{d}t &=   \varepsilon \mathrm{d}\mathcal{Q}(\varrho_\varepsilon \mathbf{u}_\varepsilon) - \varepsilon \textbf{F}_\varepsilon\mathrm{d}t,
\end{aligned}
\end{equation}
where 
\begin{align*}
\textbf{F}_\varepsilon  = \mathrm{div}\mathcal{Q}(\varrho_\varepsilon \mathbf{u}_\varepsilon\otimes \mathbf{u}_\varepsilon)-\nu\Delta\mathcal{Q} \mathbf{u}_\varepsilon -(\lambda +\nu)\nabla\mathrm{div}\mathbf{u}_\varepsilon + \frac{1}{\varepsilon^2}\nabla [\varrho_\varepsilon^\gamma - 1 -\gamma(\varrho_\varepsilon -1)].
\end{align*}
Now let us observe that from \eqref{summary}$_5$ and the continuity of $\mathcal{Q}$, we have that 
\begin{align}
\label{F0}
\mathrm{div}\mathcal{Q}(\varrho_\varepsilon  \mathbf{u}_\varepsilon \otimes \mathbf{u}_\varepsilon)\in L^p(\Omega;L^2(0,T;W_{\mathrm{loc}}^{-1,\frac{6\gamma}{4\gamma +3}}(\mathbb{R}^3)) )
\end{align}
independently of $\varepsilon$. And that
\begin{align}
\label{F1}
 \nu \Delta\mathcal{Q}\mathbf{u}_\varepsilon  +(\lambda+\nu)\nabla\mathrm{div} \mathbf{u}_\varepsilon \in L^p(\Omega;L^2(0,T;W^{-1,2}_{\mathrm{loc}}(\mathbb{R}^3)) )
\end{align}
uniformly in $\varepsilon$ by virtue of \eqref{summary}$_2$. Lastly, the choice of $a=\frac{1}{\varepsilon^2}$ and $\overline{\varrho}_\varepsilon=1$ in the pressure potential \eqref{pressurePotential} of the energy estimate \eqref{energyEx} and Taylor's theorem means that for $s:=\min\{2,\gamma\}>1$,
\begin{align}
\label{F3}
\frac{1}{\varepsilon^2}\nabla [\varrho_\varepsilon^\gamma - 1 -\gamma(\varrho_\varepsilon -1)] \in L^p(\Omega;L^\infty(0,T;W^{-1,s}_{\mathrm{loc}}(\mathbb{R}^3)) ).
\end{align}
uniformly in $\varepsilon$. cf. \eqref{pressurePotentialA} for $\overline{\varrho}=1$ and \eqref{summary}$_1$. By combining \eqref{F0}, \eqref{F1} and \eqref{F3} with the embeddings $W^{-1,\frac{6\gamma}{4\gamma+3}}(B_r)\hookrightarrow W^{-l,2}(B_r)$ and $W^{-1,s}(B_r)\hookrightarrow W^{-l,2}(B_r)$, where $B_r$ is a ball of radius $r>0$, it holds that for $l>5/2$, 
\begin{equation}
\begin{aligned}
\label{force}
\textbf{F}_\varepsilon   &\in  L^p\left(\Omega;L^2(0,T;W^{-l,2}_{\mathrm{loc}}(\mathbb{R}^3))\right)
\end{aligned}
\end{equation}  uniformly in $\varepsilon$.

\subsection{Compactness}
\label{singularLimitCompactness}
To explore compactness for the acoustic equation, let first define the path space $\chi = \chi_\varrho \times \chi_\mathbf{u} \times \chi_{\varrho \mathbf{u}} \times \chi_W$ where
\begin{align*}
\chi_\varrho &= C_\omega\left([0,T];L^\gamma_{\mathrm{loc}}(\mathbb{R}^3)\right),
&
\chi_\mathbf{u} &= \left(L^2(0,T;W^{1,2}_{\mathrm{loc}}(\mathbb{R}^3)),\omega\right),
\\
\chi_{\varrho \mathbf{u}} &= C_\omega\left([0,T];L^{\frac{2\gamma}{\gamma+1}}_{\mathrm{loc}}(\mathbb{R}^3)\right),
&
\chi_W &= C\left([0,T];\mathfrak{U}_0\right),
\end{align*}
and let
\begin{enumerate}
\item $\mu_{\varrho_\varepsilon}$ be the law of $\varrho_\varepsilon$ on the space $\chi_{\varrho}$,
\item $\mu_{\mathbf{u}_\varepsilon}$ be the law of $\mathbf{u}_\varepsilon$ on $\chi_{\mathbf{u}}$,
\item $\mu_{\mathcal{P}(\varrho_\varepsilon \mathbf{u}_\varepsilon)}$ be the law of $\mathcal{P}(\varrho_\varepsilon \mathbf{u}_\varepsilon)$ on the space $\chi_{\varrho \mathbf{u}}$,
\item $\mu_{W}$ be the law of $W$ on the space $\chi_{W}$,
\item $\mathrm{\mu}^\varepsilon$ be the joint law of $\varrho_\varepsilon,\mathbf{u}_\varepsilon,\mathcal{P}(\varrho_\varepsilon \mathbf{u}_\varepsilon)$, and $W$ on the space $\chi$.
\end{enumerate}
Then the following lemma, the proof of which is similar to \cite[Corollary 3.7]{breit2015incompressible},  holds true.

\begin{lemma}
\label{prop:tightness}
The sets 
 $\{\mu^\varepsilon; \varepsilon\in(0,1) \}$ is tight on  $\chi$.
\end{lemma}

Now similar to Proposition \ref{prop:Jakubow0}, we apply the Jakubowski--Skorokhod representation theorem \cite{jakubowski1998short} to get the following proposition.

\begin{proposition}
\label{prop:Jakubow}
There exists a subsequence $\mu^\varepsilon$ (not relabelled), a probability space $(\tilde{\Omega}, \tilde{\mathscr{F}}, \tilde{\mathbb{P}})$ with $\chi$-valued Borel measurable random variables $(\tilde{\varrho}_\varepsilon, \tilde{\mathbf{u}}_\varepsilon, \tilde{\mathbf{q}_\varepsilon}, \tilde{W}_\varepsilon)$, $n\in\mathbb{N}$, and $(\tilde{\varrho}, \tilde{\mathbf{u}}, \tilde{\mathbf{q}}, \tilde{W})$ such that
\begin{itemize}
\item the law of $(\tilde{\varrho}_\varepsilon, \tilde{\mathbf{u}}_\varepsilon, \tilde{\mathbf{q}_\varepsilon}, \tilde{W}_\varepsilon)$ is given by $\mu^\varepsilon$, $\varepsilon\in(0,1)$,
\item the law of  $(\tilde{\varrho}, \tilde{\mathbf{u}}, \tilde{\mathbf{q}}, \tilde{W})$, denoted by $\mu$ is a Randon measure,
\item $(\tilde{\varrho}_\varepsilon, \tilde{\mathbf{u}}_\varepsilon, \tilde{\mathbf{q}_\varepsilon}, \tilde{W}_\varepsilon)$ converges $\tilde{\mathbb{P}}-$a.s to $(\tilde{\varrho}, \tilde{\mathbf{u}}, \tilde{\mathbf{q}}, \tilde{W})$ in the topology of $\chi$.
\end{itemize}
\end{proposition}

To extend this new probability space $(\tilde{\Omega}, \tilde{\mathscr{F}}, \tilde{\mathbb{P}})$ into a stochastic basis, we endow it with the  $\tilde{\mathbb{P}}-$augmented canonical filtrations for $(\tilde{\varrho}_\varepsilon, \tilde{\mathbf{u}}_\varepsilon, \tilde{W}_\varepsilon)$ and  $(\tilde{\varrho}, \tilde{\mathbf{u}}, \tilde{W})$, respectively, by setting
\begin{align*}
\tilde{\mathscr{F}}^\varepsilon_t = \sigma\left( \sigma(\textbf{r}_t\tilde{\varrho}_\varepsilon,  \textbf{r}_t\tilde{\mathbf{u}}_\varepsilon,  \textbf{r}_t\tilde{W}_\varepsilon)  \cup  \{N\in\tilde{\mathscr{F}};\,\tilde{\mathbb{P}}(N)=0  \} \right), \quad
t\in[0,T],
\\
\tilde{\mathscr{F}}_t = \sigma\left( \sigma(  \textbf{r}_t\tilde{\mathbf{u}},  \textbf{r}_t\tilde{W})  \cup  \{N\in\tilde{\mathscr{F}};\,\tilde{\mathbb{P}}(N)=0  \} \right), \quad
t\in[0,T].
\end{align*}
where $\textbf{r}_t$ is the continuous function defined in \eqref{continuousFunction} above adapted to the spaces defined in this section.

\subsection{Identification of the limit}
We now verify that on this new probability space, our new processes  
$$[ (\tilde{\Omega}, \tilde{\mathscr{F}}, (\tilde{\mathscr{F}}^\varepsilon_t), \tilde{\mathbb{P}}),\tilde{\varrho}_\varepsilon, \tilde{\mathbf{u}}_\varepsilon,  \tilde{W}_\varepsilon] \text{ and } [ (\tilde{\Omega}, \tilde{\mathscr{F}}, (\tilde{\mathscr{F}}_t), \tilde{\mathbb{P}}), \tilde{\mathbf{u}},  \tilde{W}]$$
are indeed finite energy weak martingale solutions and a weak martingale solution respectively for Eqs. \eqref{comprSPDE} and \eqref{incomprSPDE}.

\begin{proposition}
$[ (\tilde{\Omega}, \tilde{\mathscr{F}}, (\tilde{\mathscr{F}}^\varepsilon_t)_{t\geq0}, \tilde{\mathbb{P}}),\tilde{\varrho}_\varepsilon, \tilde{\mathbf{u}}_\varepsilon,  \tilde{W}_\varepsilon]$ is a finite energy weak martingale solution of Eq. \eqref{comprSPDE} with initial law $\Lambda_\varepsilon$ for $\varepsilon\in(0,1)$.
\end{proposition}
The proof of this proposition is similar to \cite[Proposition 3.10]{breit2015incompressible}. 

Consequently, the uniform bounds shown in \eqref{summary}, \eqref{convergenceofDensity}  and \eqref{force} earlier hold for these corresponding random processes on this new space. In particular, we have that 
\begin{equation}
\begin{aligned}
\label{newSummary}
\tilde{\varphi}_\varepsilon  &\in  L^p\left(\Omega;L^\infty(0,T;L^{\min\{2,\gamma\}}(\mathbb{R}^3)) \right),
\\
\tilde{\textbf{F}}_\varepsilon  &\in  L^p\left(\Omega;L^2(0,T;W^{-l,2}_{\mathrm{loc}}(\mathbb{R}^3)) \right),
\\
\tilde{\mathbf{u}}_\varepsilon  &\in  L^p\left(\Omega;L^2(0,T;W^{1,2}_{\mathrm{loc}}(\mathbb{R}^3)) \right),
\\
\tilde{\varrho}_\varepsilon \tilde{\mathbf{u}}_\varepsilon  &\in L^p(\Omega; L^\infty(0,T;L^{\frac{2\gamma}{\gamma+1}}_{\mathrm{loc}}(\mathbb{R}^3)))
\end{aligned}
\end{equation}
holds uniformly in $\varepsilon$ for $p\in[1,\infty)$ and where $l>5/2$, $\tilde{\varphi}_\varepsilon=  \frac{\tilde{\varrho}_\varepsilon  -1}{\varepsilon}$ and 
\begin{align*}
\tilde{\textbf{F}}_\varepsilon= \mathrm{div}\mathcal{Q}(\tilde{\varrho}_\varepsilon \tilde{\mathbf{u}}_\varepsilon\otimes \tilde{\mathbf{u}}_\varepsilon)-\nu\Delta\mathcal{Q} \tilde{\mathbf{u}}_\varepsilon -(\lambda +\nu)\nabla\mathrm{div}\tilde{\mathbf{u}}_\varepsilon + \frac{1}{\varepsilon^2}\nabla [\tilde{\varrho}_\varepsilon^\gamma - 1 -\gamma(\tilde{\varrho}_\varepsilon -1)]
\end{align*}

We now verify that indeed the limit  process satisfies Definition \ref{def:martSolutionIncompre}. This will complete the proof of Theorem \ref{thm:one}.
\begin{proposition}
\label{prop:limitMarttingale}
$[ (\tilde{\Omega}, \tilde{\mathscr{F}}, (\tilde{\mathscr{F}}_t)_{t\geq0}, \tilde{\mathbb{P}}), \tilde{\mathbf{u}},  \tilde{W}]$ is a weak martingale solution of Eq. \eqref{incomprSPDE} with initial law $\Lambda$.
\end{proposition}
\begin{proof}
The proof of this proposition will follow from the following lemmata and propositions.
\begin{lemma}
\label{lem:M}
For all $t\in [0,T]$ and $\phi \in C^\infty_c(\mathbb{R}^3)$, we let
\begin{align*}
M(\varrho, \mathbf{u}, \mathbf{q})_t   &=  \langle \mathbf{q}(t), \phi \rangle   -  \langle \mathbf{q}(0), \phi \rangle  -  \int_0^t  \langle \mathbf{q}\otimes  \mathbf{u}, \nabla\phi \rangle\mathrm{d}s  +\nu \int_0^t  \langle \nabla \mathbf{u}, \nabla\phi \rangle\mathrm{d}s
\\
&+(\lambda+  \nu)\int_0^t  \langle \mathrm{div} \,\mathbf{u}, \mathrm{div}\,\phi \rangle\mathrm{d}s  -  \frac{1}{\varepsilon^2}\int_0^t  \langle \varrho^\gamma, \mathrm{div}\phi \rangle\mathrm{d}s.
\end{align*}
Then
$M(\tilde{\varrho}_\varepsilon,  \tilde{\mathbf{u}}_\varepsilon,  \tilde{\varrho}_\varepsilon\tilde{\mathbf{u}}_\varepsilon)_{t}  \rightarrow    M(1,  \tilde{\mathbf{u}},  \tilde{\mathbf{u}})_{t}$  $\tilde{\mathbb{P}}-$a.s. as $\varepsilon\rightarrow 0$.
\end{lemma}
\begin{proof}
The proof of this lemma  follows further from combining Proposition \ref{prop:Jakubow} with the Lemmata \ref{allStrongbutOne}, \ref{lem:expBound} and Proposition \ref{gradientPartOfMomentumConvergece} below. 

\begin{lemma}
\label{allStrongbutOne}
For every $q<6$, the following $\tilde{\mathbb{P}}-$a.s. convergence holds:
\begin{align}
(\tilde{\varrho}_\varepsilon -1)\rightarrow  0  \quad &\text{in}\quad L^\infty(0,T;L^{\min\{2,\gamma\}}(\mathbb{R}^3)), \label{strongDens}
\\
\mathcal{P}(\tilde{\varrho}_\varepsilon \tilde{\mathbf{u}}_\varepsilon)  \rightarrow \tilde{\mathbf{u}}\quad &\text{in}\quad L^2(0,T;W_{\mathrm{loc}}^{-1,2}(\mathbb{R}^3)),
\label{solenoidalPartOfMomentumConvergece} \\
\mathcal{P}\tilde{\mathbf{u}}_\varepsilon\rightarrow \tilde{\mathbf{u}}\quad &\text{in}\quad L^2(0,T;L_{\mathrm{loc}}^q(\mathbb{R}^3)).
\label{strongSolenoidalvelocity}
\end{align}
\end{lemma}
\begin{proof}
See \cite{breit2015incompressible}.
\end{proof}

\begin{remark}
Henceforth, we write `$\lesssim$' for `$\leq\,c$' and `$\eqsim$' for `$=\,c$' where $c$, which may varies from line to line is some universal constant that is independent of $\varepsilon$ but may depend on other variables.
\end{remark}

\begin{proposition}
\label{gradientPartOfMomentumConvergece}
The strong convergence below holds.
\begin{align*}
\mathcal{Q}(\tilde{\varrho}_\varepsilon \tilde{\mathbf{u}}_\varepsilon)  \rightarrow 0\quad \text{in}\quad L^2(0,T;L^\frac{2\gamma}{\gamma+1}_{\mathrm{loc}}(\mathbb{R}^3))\quad\tilde{\mathbb{P}}-a.s.
\end{align*}
\end{proposition}
\begin{proof}
Let define the function $\tilde{\Psi}_\varepsilon  =  \Delta^{-1}\mathrm{div}(\tilde{\varrho}_\varepsilon \tilde{\mathbf{u}}_\varepsilon )$ such that $\nabla \tilde{\Psi}_\varepsilon  =\mathcal{Q}(\tilde{\varrho}_\varepsilon \tilde{\mathbf{u}}_\varepsilon)$. Then equation \eqref{acousticSPD} becomes
\begin{equation}
\begin{aligned}
\label{acousticSPD1}
\varepsilon\mathrm{d}( \tilde{\varphi}_\varepsilon) +  \Delta \tilde{\Psi}_\varepsilon\,\mathrm{d}t   &=  0,   \\
\varepsilon \mathrm{d}\nabla\tilde{\Psi}_\varepsilon  +  \gamma \nabla \tilde{\varphi}_\varepsilon\,\mathrm{d}t   &=    \varepsilon \tilde{\textbf{F}}_\varepsilon\,\mathrm{d}t  + \varepsilon \mathcal{Q}\Phi(\tilde{\varrho}_\varepsilon, \tilde{\varrho}_\varepsilon \tilde{\mathbf{u}}_\varepsilon)\mathrm{d}\tilde{W}_\varepsilon.
\end{aligned}
\end{equation}
We however observe that Eq. \eqref{acousticSPD1} is equivalent to 
\begin{equation}
\begin{aligned}
\label{acousticSPD111}
\varepsilon\,\mathrm{d}
\begin{bmatrix}
       \tilde{\varphi}_\varepsilon        \\[0.3em]
       \nabla\tilde{\Psi}_\varepsilon
\end{bmatrix}  
=
\mathcal{A}
\begin{bmatrix}
       \tilde{\varphi}_\varepsilon        \\[0.3em]
       \nabla\tilde{\Psi}_\varepsilon
\end{bmatrix}  \mathrm{d}t 
+
\varepsilon\,
\begin{bmatrix}
       0        \\[0.3em]
       \tilde{\textbf{F}}_\varepsilon
\end{bmatrix}  \mathrm{d}t
+
\varepsilon\,
\begin{bmatrix}
       0        \\[0.3em]
      \mathcal{Q}\Phi
\end{bmatrix}  \mathrm{d}\tilde{W}_\varepsilon
\end{aligned}
\end{equation}
where the usual wave operator
\begin{align}
\mathcal{A}= 
\begin{bmatrix}
       0 & -\mathrm{div}       \\[0.3em]
       -\gamma\nabla & 0
\end{bmatrix}
\end{align}
is an infinitesimal generator of  a strongly continuous semigroup $S(\cdot)=\exp(\mathcal{A}\cdot)$. See for example \cite{desjardins1999low}. Also since $\Phi:=\Phi(\tilde{\varrho},\tilde{\varrho}\tilde{\mathbf{u}})$ is the Hilbert--Schmidt operator and equation \eqref{acousticSPD1} is satisfied weakly in the probabilistic sense, it follows that this weak solution is also a mild solution. See for example   \cite[Theorem 6.5]{da2014stochastic}. As such after rescaling, we obtain the mild equation
\begin{equation}
\begin{aligned}
\label{duhamel}
\begin{bmatrix}
       \tilde{\varphi}_\varepsilon        \\[0.3em]
       \nabla\tilde{\Psi}_\varepsilon
\end{bmatrix}  (t) 
&=
S\left(\frac{t}{\varepsilon}\right)
\begin{bmatrix}
       \tilde{\varphi}_\varepsilon(0)        \\[0.3em]
       \nabla\tilde{\Psi}_\varepsilon(0)
\end{bmatrix} 
+\int_0^t S\left(\frac{t-s}{\varepsilon}\right)
\begin{bmatrix}
       0        \\[0.3em]
\tilde{\textbf{F}}_\varepsilon
\end{bmatrix}
\mathrm{d}s
\\
&+\int_0^t S\left(\frac{t-s}{\varepsilon}\right)
\begin{bmatrix}
       0     \\[0.3em]
       \mathcal{Q}\tilde{\Phi}_\varepsilon
\end{bmatrix}
\mathrm{d}\tilde{W}_{s,\varepsilon}
\end{aligned}
\end{equation}
where the semigroup $S(t)$ is such that
\begin{equation}
\label{semigroup}
S\left(t\right)
\begin{bmatrix}
       \tilde{\varphi}_0        \\[0.3em]
       \nabla\tilde{\Psi}_0
\end{bmatrix} 
=
\begin{bmatrix}
       \tilde{\varphi}        \\[0.3em]
       \nabla\tilde{\Psi}
\end{bmatrix}  (t) 
\end{equation}
is the solution to the homogeneous problem
\begin{equation}
\begin{aligned}
\label{homogeAcoustic}
\mathrm{d}( \tilde{\varphi}) +  \Delta \tilde{\Psi}\,\mathrm{d}t   &=  0,   \\
 \mathrm{d}\nabla\tilde{\Psi}  +  \gamma \nabla \tilde{\varphi}\,\mathrm{d}t   &=    0,
 \\
\tilde{\varphi}(0) = \tilde{\varphi}_0; \quad \nabla\tilde{\Psi}(0) &= \nabla\tilde{\Psi}_0.
\end{aligned}
\end{equation}
Using Fourier transforms (in space), we obtain solution of Eq. \eqref{homogeAcoustic} which is given by the pair
\begin{equation}
\begin{aligned}
\label{solution}
\nabla\tilde{\Psi}(t,x)  &=  \frac{e^{i\sqrt{-\gamma\Delta}t}}{2}\left( \nabla\tilde{\Psi}_0(x)  -\frac{i\sqrt{\gamma}}{\sqrt{-\Delta}} \tilde{\varphi}_0(x) \right)    
\\
&+   \frac{e^{-i\sqrt{-\gamma\Delta}t}}{2}\left( \nabla\tilde{\Psi}_0(x)  +  \frac{i\sqrt{\gamma}}{\sqrt{-\Delta}} \tilde{\varphi}_0(x)\right),
\\
\tilde{\varphi}(t,x)  &=  \frac{e^{i\sqrt{-\gamma\Delta}t}}{2}\left( \frac{i\sqrt{-\Delta}}{\sqrt{\gamma}}\nabla\tilde{\Psi}_0(x)  + \tilde{\varphi}_0(x) \right)   
\\
&-   \frac{e^{-i\sqrt{-\gamma\Delta}t}}{2}\left(\frac{i\sqrt{-\Delta}}{\sqrt{\gamma}} \nabla\tilde{\Psi}_0(x)  -  \tilde{\varphi}_0(x)\right) .
\end{aligned}
\end{equation}

%
The lemma below is crucial to the proof of Proposition \ref{gradientPartOfMomentumConvergece} and  is an adaptation of \cite[Lemma 2.2]{smith2000global} to our setting. cf. \cite[Lemma 3.1]{feireisl2012multi}.
\begin{lemma}
\label{lem:expBound}
Let $\phi(x)\in C^\infty_c(\mathbb{R}^3)$, we have
\begin{align*}
 \int_{\mathbb{R}}\Vert  e^{i\sqrt{-\gamma \Delta}t}[\mathbf{v}\phi] \Vert^2_{L^2(\mathbb{R}^3)}\mathrm{d}t  \,\leq c(\phi)\,\Vert \mathbf{v} \Vert^2_{L^2(\mathbb{R}^3)}
\end{align*}
for any $\mathbf{v} \in L^2(\mathbb{R}^3)$.
\end{lemma}
\begin{proof}
For simplicity, we assume that $\gamma=1$. General $\gamma>1$ will then follow by rescaling $\delta$ below.

Using Plancherel’s theorem in $t$ and $x$, we have that
\begin{align*}
 \int_{\mathbb{R}}&\Vert  e^{i\sqrt{- \Delta}t}[\mathbf{v}\phi] \Vert^2_{L^2(\mathbb{R}^3)}\mathrm{d}t  =c(\pi)
 \int_{\mathbb{R}}\int_{\mathbb{R}^3}\Big\vert\int_{\mathbb{R}^3} \widehat{\phi}(\xi-\eta)\delta(\tau-{\vert\eta\vert})\widehat{\mathbf{v}}(\eta)\,\mathrm{d}\eta\Big\vert^2\mathrm{d}\xi\mathrm{d}\tau 
 \\
 &=c(\pi)
 \int_{\mathbb{R}}\int_{\mathbb{R}^3}\Big\vert\int_{\{\tau = \vert\eta\vert\}} \widehat{\phi}(\xi-\eta)\widehat{\mathbf{v}}(\eta)\,\mathrm{d}S_\eta\Big\vert^2\mathrm{d}\xi\mathrm{d}\tau 
 \\
 &\leq c(\pi)
 \int_{\mathbb{R}}\int_{\mathbb{R}^3}\Big(\int_{\{\tau = \vert\eta\vert\}}\vert \widehat{\phi}(\xi-\eta)\vert\,\mathrm{d}S_\eta\Big)
 \Big(\int_{\{\tau = \vert\eta\vert\}} \vert\widehat{\phi}(\xi-\eta)\vert\vert\widehat{\mathbf{v}}(\eta)\vert^2\,\mathrm{d}S_\eta\Big)\mathrm{d}\xi\mathrm{d}\tau 
 \\
  &\leq c(\pi,\phi)
 \int_{\mathbb{R}^3}\int_{\mathbb{R}}\int_{\{\tau = \vert\eta\vert\}} \vert\widehat{\phi}(\xi-\eta)\vert\vert\widehat{\mathbf{v}}(\eta)\vert^2\,\mathrm{d}S_\eta\mathrm{d}\tau\mathrm{d}\xi 
  \\
  &\leq c(\pi,\phi)
 \int_{\mathbb{R}^3}\int_{\mathbb{R}^3} \vert\widehat{\phi}(\xi-\eta)\vert\vert\widehat{\mathbf{v}}(\eta)\vert^2\,\mathrm{d}\eta\mathrm{d}\xi 
\leq c(\pi,\phi)\,\Vert \mathbf{v} \Vert^2_{L^2(\mathbb{R}^3)}
\end{align*}
where we have used the Cauchy-Schwartz inequality.
\end{proof}

Moving on, we now consider a smooth cut-off function (with expanding support) $\eta_r\in C^\infty_0(B_{2r})$ with $\eta_r\equiv 1$ in $B_r$ for $r>0$ and zero  elsewhere.
We now mollify the product of this cut-off function and our functions in \eqref{acousticSPD1} by means of spatial convolution with the standard mollifier. That is, if $v$ is one of the functions in \eqref{acousticSPD1}, we set 
$$v^\kappa =(\eta_r v)\ast \varphi^\kappa$$
where $\varphi^\kappa$ is the standard mollifier. This we do to ensure that the regularized functions are globally integrable.
First off, we note that since \eqref{newSummary}$_4$ holds uniformly in $\varepsilon$, for an arbitrary small $\delta>0$, we can find a  $\kappa(\delta)$  such that
\begin{align}
\label{mollification}
\tilde{\mathbb{E}}\sup_{t\in[0,T]}\Vert (\tilde{\varrho}_\varepsilon  \tilde{\mathbf{u}}_\varepsilon)^\kappa  -  \tilde{\varrho}_\varepsilon   \tilde{\mathbf{u}}_\varepsilon \Vert^p_{L^{\frac{2\gamma}{\gamma+1}}(B)}
\leq \delta
\end{align}
for any $1\leq p <\infty$ and an arbitrary ball $B\subset\subset B_r$ for $r>0$. 
Then using  \eqref{semigroup} , \eqref{solution} and Lemma \ref{lem:expBound}, we obtain
\begin{equation}
\begin{aligned}
\label{seminormEst}
\tilde{\mathbb{E}}\, \left\Vert  S(t)
\begin{bmatrix}
       \tilde{\varphi}^\kappa_0        \\[0.3em]
       \nabla\tilde{\Psi}^\kappa_0
\end{bmatrix}  \right\Vert^2 _{L^2(\mathbb{R}\times B)} 
\leq c_{h,\gamma} \, 
\tilde{\mathbb{E}}\,\left\Vert  
\begin{bmatrix}
       \tilde{\varphi}^\kappa_0        \\[0.3em]
       \nabla\tilde{\Psi}^\kappa_0
\end{bmatrix}   \right\Vert^2_{L^2(\mathbb{R}^3)},
\end{aligned}
\end{equation}
for any ball $B\subset\mathbb{R}^3$ and where in particular, the constant is independent of $\kappa$. So by rescaling in time, i.e, setting $s=\frac{t}{\varepsilon}$ so that $\mathrm{d}s=\frac{\mathrm{d}t}{\varepsilon}$, we get
\begin{equation}
\begin{aligned}
\label{seminormEst1}
\tilde{\mathbb{E}}\,\left\Vert  S\left(\frac{t}{\varepsilon}\right)
\begin{bmatrix}
       \tilde{\varphi}^\kappa_\varepsilon(0)        \\[0.3em]
       \nabla\tilde{\Psi}^\kappa_\varepsilon(0)
\end{bmatrix}  \right\Vert^2_{L^2((0,T)\times B)}  
\,&\leq 
\tilde{\mathbb{E}}\, \left\Vert  S\left(\frac{t}{\varepsilon}\right)
\begin{bmatrix}
       \tilde{\varphi}^\kappa_\varepsilon(0)        \\[0.3em]
       \nabla\tilde{\Psi}^\kappa_\varepsilon(0)
\end{bmatrix}   \right\Vert^2_{L^2(\mathbb{R}\times B)} 
\\
&\lesssim \,\varepsilon\, \tilde{\mathbb{E}}\,
\left\Vert  
\begin{bmatrix}
        \tilde{\varphi}^\kappa_\varepsilon(0)        \\[0.3em]
       \nabla\tilde{\Psi}^\kappa_\varepsilon(0)
\end{bmatrix}   \right\Vert^2_{L^2(\mathbb{R}^3)} 
\end{aligned}
\end{equation}
with a constant that is independent of $\varepsilon$. Now by the continuity of $\mathcal{Q}$,  \eqref{mollification}, and the initial law defined in the statement of Theorem \ref{thm:dissi},
 we conclude that
\begin{equation}
\begin{aligned}
\label{firstEst}
\tilde{\mathbb{E}}\Bigg\Vert  S\left(\frac{t}{\varepsilon}\right)
\begin{bmatrix}
       \tilde{\varphi}^\kappa_\varepsilon(0)        \\[0.3em]
       \nabla\tilde{\Psi}^\kappa_\varepsilon(0)
\end{bmatrix}  &\Bigg\Vert^2_{L^2((0,T)\times B)}  
\lesssim \,\varepsilon\,\tilde{\mathbb{E}}\Bigg( \Vert \tilde{\varphi}^\kappa_\varepsilon(0) \Vert^2_{L^{\min\{2,\gamma\}}(\mathbb{R}^3)}   
\\
&+  \Vert \tilde{\mathbf{q}}^\kappa_\varepsilon(0) \Vert^2_{L^{\frac{2\gamma}{\gamma+1}}(\mathbb{R}^3)}\Bigg)
\leq\,\varepsilon\, c_{\kappa, M}.
\end{aligned}
\end{equation}

Similarly we have that for any ball $B\subset \mathbb{R}^3$,
\begin{equation}
\begin{aligned}
\label{secondEst}
\tilde{\mathbb{E}}\,  \Bigg\Vert \int\limits_0^t  S\left(\frac{t-s}{\varepsilon}\right)
       \tilde{\textbf{F}}^\kappa_\varepsilon
\mathrm{d}s\,   &\Bigg\Vert^2_{L^2((0,T)\times B)}
\, \leq \, 
\tilde{\mathbb{E}}\,\left\Vert   S\left(\frac{t-s}{\varepsilon}\right)
       \tilde{\textbf{F}}^\kappa_\varepsilon
   \right\Vert^2 _{L^2((0,t)\times(0,T)\times B)}  
\\
&\leq \,
\tilde{\mathbb{E}}\, \left\Vert S\left(\frac{t}{\varepsilon}\right) S\left(\frac{-s}{\varepsilon}\right)
       \tilde{\textbf{F}}^\kappa_\varepsilon
  \right\Vert^2_{L^2(\mathbb{R}\times(0,T)\times B)}  
\\
&\leq \varepsilon\,c_{\gamma} \, \tilde{\mathbb{E}}\,
\left\Vert 
S\left(\frac{-s}{\varepsilon}\right)
       \tilde{\textbf{F}}^\kappa_\varepsilon
  \right\Vert^2_{L^2((0,T)\times B)} 
\\
&\eqsim \,\varepsilon  \, \tilde{\mathbb{E}}\,
\left\Vert 
       \tilde{\textbf{F}}^\kappa_\varepsilon
  \right\Vert^2_{L^2((0,T)\times \mathbb{R}^3)}
\leq \varepsilon \, c_{\gamma,\kappa} 
\end{aligned}
\end{equation} 
Where we have used Jensen's inequality and Fubini's theorem in the first inequality, extended $(0,t)$ to $\mathbb{R}$ and used the semigroup property in the second inequality, applied similar reasoning as in \eqref{seminormEst1} in the third inequality and then used that $\left(S(t)\right)_t$ is a group of isometries on $L^2$ (extended by zero outside of the ball) in the last line above.

We have therefore obtained the following bounds
\begin{equation}
\begin{aligned}
\label{combinedFirstandSecond}
\tilde{\mathbb{E}}\,\left\Vert  S\left(\frac{t}{\varepsilon}\right)
\begin{bmatrix}
       \tilde{\varphi}^\kappa_\varepsilon(0)        \\[0.3em]
       \nabla\tilde{\Psi}^\kappa_\varepsilon(0)
\end{bmatrix}  \right\Vert^2_{L^2(0,T;L^2(B))}  
&\lesssim \,\varepsilon, \quad 
\\
\tilde{\mathbb{E}}\, \left\Vert \int\limits_0^t  S\left(\frac{t-s}{\varepsilon}\right)
\begin{bmatrix}
       0       \\[0.3em]
       \tilde{\textbf{F}}^\kappa_\varepsilon
\end{bmatrix}\mathrm{d}s\,   \right\Vert^2_{L^2(0,T;L^2(B))}  &\lesssim \,\varepsilon
\end{aligned}
\end{equation}
for any ball $B\subset\mathbb{R}^3$. Now let make the notation $\tilde{\Phi}_\varepsilon^\kappa(e_i):  = g_i\left(\cdot, \tilde{\varrho}_\varepsilon(\cdot), (\tilde{\mathbf{q}}_\varepsilon )(\cdot) \right)^\kappa=: \tilde{g}_i^{\varepsilon,\kappa}$. We  notice that for a continuous function $S(t)$ and a continuous operator $\mathcal{Q}$, the quantity $S(t)\mathcal{Q}\Phi$ is Hilbert--Schmidt if $\Phi$ is Hilbert--Schmidt. As such, it follows from It\'{o} isometry that
\begin{align*}
\tilde{\mathbb{E}}\,\Bigg\Vert
\int\limits_0^t 
 &S\left(\frac{t-s}{\varepsilon}\right)
\begin{bmatrix}
       0       \\[0.3em]
       \mathcal{Q}\tilde{\Phi}^\kappa_\varepsilon
\end{bmatrix} \,\mathrm{d}\tilde{W}_{\varepsilon}(s) \Bigg\Vert^2_{L^2((0,T)\times B)} 
\\
&=
\tilde{\mathbb{E}}\,\int\limits_0^t \left\Vert 
 S\left(\frac{t-s}{\varepsilon}\right)
       \mathcal{Q}\tilde{\Phi}^\kappa_\varepsilon
       \right\Vert^2_{L_2(\mathfrak{U};L^2((0,T)\times B))}\mathrm{d}s
 \\
 &=
\tilde{\mathbb{E}}\,\int\limits_0^t \sum_{i\in\mathbb{N}} \left\Vert 
 S\left(\frac{t-s}{\varepsilon}\right)
       \mathcal{Q}\tilde{g}_i^{\varepsilon,\kappa}
       \right\Vert^2_{L^2((0,T)\times B)}\mathrm{d}s
\\
&\lesssim
\int\limits_0^T
\sum\limits_{i\in\mathbb{N}}  \int_{\mathbb{R}}\tilde{\mathbb{E}}
\left\Vert 
 S\left(\frac{t-s}{\varepsilon}\right)
       \mathcal{Q}\tilde{g}_i^{\varepsilon,\kappa}
          \right\Vert^2_{L^2(B)}
\mathrm{d}s\,\mathrm{d}t
\end{align*} 
where the above involved extending $s$ from $(0,t)$ to $\mathbb{R}$ as well as Fubini's theorem. 

Now using the semigroup property and similar estimate as in equation \eqref{seminormEst} and \eqref{seminormEst1}, followed by the fact that the semigroup is an isometry with respect to the $L^2$-norm, we get that

\begin{align*}
\int\limits_0^T
\sum\limits_{i\in\mathbb{N}}   \int_{\mathbb{R}}\tilde{\mathbb{E}}
&\Big\Vert 
 S\left(\frac{t-s}{\varepsilon}\right)
       \mathcal{Q}\tilde{g}_i^{\varepsilon,\kappa}
          \Big\Vert^2_{L^2(B)}
\mathrm{d}s\mathrm{d}t
\\
&=
\int\limits_0^T
\sum\limits_{i\in\mathbb{N}}  \tilde{\mathbb{E}}
\left\Vert 
 S\left(\frac{t}{\varepsilon}\right)S\left(\frac{-s}{\varepsilon}\right)
       \mathcal{Q}\tilde{g}_i^{\varepsilon,\kappa}
          \right\Vert^2_{L^2(\mathbb{R}\times B)}
\mathrm{d}t
 \\
&\lesssim\varepsilon\,
\int\limits_0^T
\sum\limits_{i\in\mathbb{N}}   \tilde{\mathbb{E}}
\left\Vert 
S\left(\frac{-s}{\varepsilon}\right)
       \mathcal{Q}\tilde{g}_i^{\varepsilon,\kappa}
          \right\Vert^2_{L^2(B)}
\mathrm{d}t
\\
&\eqsim\varepsilon\,
\int\limits_0^T
\sum\limits_{i\in\mathbb{N}}   \tilde{\mathbb{E}}
\left\Vert 
       \mathcal{Q}\tilde{g}_i^{\varepsilon,\kappa}
          \right\Vert^2_{L^2(\mathbb{R}^3)}
\mathrm{d}t
\lesssim\varepsilon\,
\int\limits_0^T
\sum\limits_{i\in\mathbb{N}}   \tilde{\mathbb{E}}
\left\Vert 
     \tilde{g}_i^{\varepsilon,\kappa}
          \right\Vert^2_{L^2(\mathbb{R}^3)}
\mathrm{d}t
\\
&\lesssim\varepsilon\,
\tilde{\mathbb{E}} \int\limits_0^T
 \sum\limits_{i\in\mathbb{N}}   
\left\Vert 
      \tilde{g}_i^{\varepsilon}
          \right\Vert^2_{L^1(\mathbb{R}^3)}
\mathrm{d}t  
\lesssim \varepsilon.
\end{align*}
The last inequality follows because the noise term is assumed to be compactly supported in $\mathbb{R}^3$. See \eqref{noiseSupport}. We have therefore shown that
\begin{align*}
\tilde{\mathbb{E}}\,\Bigg\Vert
\int\limits_0^t 
 S\left(\frac{t-s}{\varepsilon}\right)
       \mathcal{Q}\tilde{\Phi}^\kappa_\varepsilon
&\mathrm{d}\tilde{W}_{\varepsilon}(s) \Bigg\Vert^2_{L^2((0,T)\times B)} \leq \varepsilon\, c_{h,\gamma, \kappa} 
\end{align*}
where the constant is independent of $\varepsilon$. Combining this with the estimates from \eqref{combinedFirstandSecond}, we get from \eqref{duhamel} that
\begin{align*}
\tilde{\mathbb{E}}\, \left\Vert     
\begin{bmatrix}
       \tilde{\varphi}_\varepsilon(t)      \\[0.3em]
       \nabla\tilde{\Psi}_\varepsilon(t)
\end{bmatrix} \right\Vert^2_{L^2((0,T)\times B)}
&= \tilde{\mathbb{E}}\,\Vert \tilde{\varphi}_\varepsilon(t) \Vert^2_{L^2((0,T)\times B)}
+ \tilde{\mathbb{E}}\,\Vert \nabla\tilde{\Psi}_\varepsilon(t) \Vert^2_{L^2((0,T)\times B)}
\\
&\lesssim I_1  + I_2  +  I_3 \,\leq \, \varepsilon \, c_{h,\gamma, \kappa}.
\end{align*}
where we have set
\begin{align*}
I_1 &:=  \tilde{\mathbb{E}}  \, \left\Vert  S\left(\frac{t}{\varepsilon}\right)
\begin{bmatrix}
       \tilde{\varphi}_\varepsilon(0)        \\[0.3em]
       \nabla\tilde{\Psi}_\varepsilon(0)
\end{bmatrix}   \right\Vert^2_{L^2((0,T)\times B)}  
\, \,
I_2 :=  \tilde{\mathbb{E}}  \,\left\Vert \int\limits_0^t  S\left(\frac{t-s}{\varepsilon}\right)
       \tilde{\textbf{F}}_\varepsilon
\mathrm{d}s\,   \right\Vert^2_{L^2((0,T)\times B)}
\\
I_3 &:=  \tilde{\mathbb{E}} \,
 \left\Vert \int\limits_0^t  S\left(\frac{t-s}{\varepsilon}\right)
       \mathcal{Q}\Phi
\mathrm{d}W_{s,\varepsilon}  \right\Vert^2_{L^2((0,T)\times B)}
\end{align*}
So in particular,
\begin{align}
\label{expExtimate1}
\tilde{\mathbb{E}}\,\Vert \nabla\tilde{\Psi}^\kappa_\varepsilon(t) \Vert^2_{L^2((0,T)\times B)}
&\leq  \varepsilon \, c_{h,\gamma, \kappa}
\end{align}
holds for any ball $B\subset \mathbb{R}^3$. We also deduce from Eq. \eqref{mollification} together with the embedding $L^\infty(0,T; L^r(B))\hookrightarrow L^2(0,T; L^r(B))$ where $r= \frac{2\gamma}{\gamma+1}$, and the continuity of $\mathcal{Q}$ that 
\begin{equation}
\begin{aligned}
\label{expExtimate2}
\tilde{\mathbb{E}}\Vert\nabla\tilde{\Psi}_\varepsilon^\kappa -  \nabla\tilde{\Psi}_\varepsilon\Vert^2_{L^2(0,T;L^r(B))}
&\leq c_{\delta,t}\, ,\quad
\tilde{\mathbb{E}}\Vert\tilde{\mathbf{q}}_\varepsilon^\kappa -  \tilde{\mathbf{q}}_\varepsilon\Vert^2_{L^2(0,T; L^r(B))}
\leq  c_{\delta,t}
\end{aligned}
\end{equation}
where $\delta$ is the arbitrarily constant from \eqref{mollification} which is independent of $\kappa$ and $\varepsilon$. As such, the constant $c_{\delta,t}$ can be made arbitrarily small for an arbitrary choice of $\delta$ so that
\begin{align*}
\lim\limits_{\kappa\downarrow0}\tilde{\mathbb{E}}\Vert\nabla\tilde{\Psi}_\varepsilon^\kappa -  \nabla\tilde{\Psi}_\varepsilon\Vert^2_{L^2(0,T; L^r(B))} = 0, \quad r=\frac{2\gamma}{\gamma+1}.
\end{align*}
Thus,  it follows from \eqref{expExtimate1} and the uniform bound \eqref{expExtimate2} that we may exchange the order of taking limits in \eqref{expExtimate2}. As such for any ball $B\subset\mathbb{R}^3$, we have that
\begin{equation}
\begin{aligned}
\label{something}
0 &\leq \lim_{\varepsilon\downarrow 0}\tilde{\mathbb{E}} \Vert \nabla\tilde{\Psi}_\varepsilon \Vert^2_{L^2(0,T; L^r(B))} 
= \lim_{\kappa\downarrow 0}\lim_{\varepsilon\downarrow 0} \tilde{\mathbb{E}}\Vert \nabla\tilde{\Psi}_\varepsilon \Vert^2_{L^2(0,T; L^r(B))}
\\
&\leq 2\lim_{\varepsilon\downarrow 0}\lim_{\kappa\downarrow 0} \tilde{\mathbb{E}}\Vert \nabla\tilde{\Psi}_\varepsilon^\kappa -  \nabla\tilde{\Psi}_\varepsilon\Vert^2_{L^2(0,T; L^r(B))}
+ 2\lim_{\kappa\downarrow 0}\lim_{\varepsilon\downarrow 0} \tilde{\mathbb{E}}\Vert \nabla\tilde{\Psi}_\varepsilon^\kappa \Vert^2_{L^2(0,T; L^r(B))}
\\
&\leq c\, \left(\lim_{\kappa\downarrow 0} \tilde{\mathbb{E}}\Vert \nabla\tilde{\Psi}_\varepsilon^\kappa -  \nabla\tilde{\Psi}_\varepsilon\Vert^2_{L^2(0,T; L^r(B))}
+ \lim_{\varepsilon\downarrow 0} \tilde{\mathbb{E}}\Vert \nabla\tilde{\Psi}_\varepsilon^\kappa \Vert^2_{L^2((0,T)\times B)} \right)
\,= \, 0
\end{aligned}
\end{equation}
hence our claim.
\end{proof}

\begin{remark}
We observe that by combining \eqref{solenoidalPartOfMomentumConvergece} and Proposition \ref{gradientPartOfMomentumConvergece}, we can only conclude that
\begin{align}
\label{strongMomemConv}
\tilde{\varrho}_\varepsilon\tilde{\mathbf{u}}_\varepsilon \rightarrow \tilde{\mathbf{u}} \quad \text{in}\quad L^2(0,T;W_{\mathrm{loc}}^{-1,2}(\mathbb{R}^3))
\end{align}
$\tilde{\mathbb{P}}-$a.s. 
\end{remark}
However, we can improve this spatial regularity.  We give this as part of the lemma below.

\begin{lemma}
\label{momentumStrong}
Let $\gamma>\frac{3}{2}$, $q<6$  and $l>\frac{3}{2}$. Then for all $r\in(\frac{3}{2},6)$, we have that
\begin{align}
\mathrm{div}(\tilde{\varrho}_\varepsilon\tilde{\mathbf{u}}_\varepsilon\otimes  \tilde{\mathbf{u}}_\varepsilon)   \rightharpoonup \mathrm{div}(\tilde{\mathbf{u}}\otimes\tilde{\mathbf{u}})\quad &\text{in}\quad L^1(0,T; W_{\mathrm{div}}^{-l,2}(B)) \label{convectiveConve},
\\
\tilde{\varrho}_\varepsilon\tilde{\mathbf{u}}_\varepsilon \rightarrow \tilde{\mathbf{u}} \quad &\text{in}\quad L^2(0,T;L^r(B))
\label{strongMomemConv1}
\end{align}
$\tilde{\mathbb{P}}-$a.s. for any ball $B\subset \mathbb{R}^3$.
\end{lemma}
\begin{proof}
To avoid repetition, we refer the reader to \cite[Proposition 3.13]{breit2015incompressible} for the proof of \eqref{convectiveConve}. However we proof \eqref{strongMomemConv1} below.

By using the identity $\mathcal{P}(\tilde{\varrho}_\varepsilon \tilde{\mathbf{u}}_\varepsilon) = \mathcal{P}(\tilde{\varrho}_\varepsilon -1)\tilde{\mathbf{u}}_\varepsilon  + \mathcal{P}\tilde{\mathbf{u}}_\varepsilon$, the reverse triangle inequality and then the triangle inequality, we have that for any ball $B\subset\mathbb{R}^3$,
\begin{align*}
&\Big\vert \Vert \mathcal{P}(\tilde{\varrho}_\varepsilon\tilde{\mathbf{u}}_\varepsilon) \Vert_{L^2(0,T;L^r(B))}  -\left\Vert \tilde{\mathbf{u}} \right\Vert_{L^2(0,T;L^r(B))}  \Big\vert
\\
&\leq 
\left\Vert \mathcal{P}(\tilde{\varrho}_\varepsilon -1)\tilde{\mathbf{u}}_\varepsilon  + \mathcal{P}\tilde{\mathbf{u}}_\varepsilon - \tilde{\mathbf{u}} \right\Vert_{L^2(0,T;L^r(B))} 
\\
&\leq
\left\Vert \mathcal{P}(\tilde{\varrho}_\varepsilon -1)\tilde{\mathbf{u}}_\varepsilon \right\Vert_{L^2(0,T;L^r(B))} + \left\Vert \mathcal{P}\tilde{\mathbf{u}}_\varepsilon - \tilde{\mathbf{u}} \right\Vert_{L^2(0,T;L^r(B))} 
\\
&\leq c
\left\{
\left\Vert \tilde{\varrho}_\varepsilon -1 \right\Vert_{L^\infty(0,T; L^{\min\{2,\gamma\}}(\mathbb{R}^3))}\left\Vert \tilde{\mathbf{u}}_\varepsilon \right\Vert_{L^2(0,T ; L^\frac{r\gamma}{\gamma-r}(B))} + \left\Vert \mathcal{P}\tilde{\mathbf{u}}_\varepsilon - \tilde{\mathbf{u}} \right\Vert_{L^2(0,T ; L^q(B))} \right\}
\\
&\rightarrow 0
\end{align*}
where we have used \eqref{newSummary}$_3$, \eqref{strongDens}, \eqref{strongSolenoidalvelocity} and the continuity of $\mathcal{P}$.

Combining this with Proposition \ref{gradientPartOfMomentumConvergece} finishes the proof.
\end{proof}
By combining \eqref{strongDens} 
with Lemma \ref{momentumStrong} we finish the proof of Lemma \ref{lem:M}.
\end{proof}

The following lemma now completes the proof of Proposition \ref{prop:limitMarttingale}.
\begin{lemma}
\label{lem:N}
For all $t\in [0,T]$ and $\phi \in C^\infty_c(\mathbb{R}^3)$, we define
\begin{align*}
N(\varrho,   \mathbf{q})_{t}    =  \sum_{k\in\mathbb{N}}\int_0^t  \langle g_k(\varrho, \mathbf{q}), \phi \rangle^2\mathrm{d}s
,\quad
N_k(\varrho,   \mathbf{q})_{t}      =  \int_0^t  \langle g_k(\varrho, \mathbf{q}), \phi \rangle\mathrm{d}s.
\end{align*}
Then we have that for $\varepsilon\in(0,1)$
\begin{align*}
N(\tilde{\varrho}_\varepsilon,   \tilde{\varrho}_\varepsilon\tilde{\mathbf{u}}_\varepsilon)_{t}  \,\rightarrow \,   N(1,    \tilde{\mathbf{u}})_{t}  \quad \tilde{\mathbb{P}}-a.s.,
\\
N_k(\tilde{\varrho}_\varepsilon,   \tilde{\varrho}_\varepsilon\tilde{\mathbf{u}}_\varepsilon)_{t}  \rightarrow    N_k(1,    \tilde{\mathbf{u}})_{t}  \quad \tilde{\mathbb{P}}-a.s.
\end{align*}
as $\varepsilon\rightarrow 0$.
\end{lemma}
\begin{proof}
By Minkowski's inequality, we have that
\begin{align*}
&\Vert \langle \Phi(\tilde{\varrho}_\varepsilon, \tilde{\varrho}_\varepsilon\tilde{\mathbf{u}}_\varepsilon)\cdot  , \phi\rangle  -  \langle \Phi(1,\tilde{\mathbf{u}})\cdot,\phi \rangle \Vert_{L_2(\mathfrak{U};\mathbb{R})} 
\\
&= \left(\sum_{k\in\mathbb{N}}\left\vert \langle \left(\Phi(\tilde{\varrho}_\varepsilon, \tilde{\varrho}_\varepsilon\tilde{\mathbf{u}}_\varepsilon)-   \Phi(1,\tilde{\mathbf{u}})\right) (e_k)\, ,\,\phi \rangle   \right\vert^2\right)^\frac{1}{2}
\\
&\leq c(\phi)\, \left(\sum_{k\in\mathbb{N}}\left\vert \int_{\mathrm{supp}(\phi)}  ( g_k(\tilde{\varrho}_\varepsilon,\tilde{\varrho}_\varepsilon\tilde{\mathbf{u}}_\varepsilon)  - g_k(1,\tilde{\mathbf{u}}))\,\mathrm{d}x   \right\vert^2 \right)^\frac{1}{2}
\\
&\leq  c\, \int_{\mathrm{supp}(\phi)}\left( \sum_{k\in\mathbb{N}}\left\vert  g_k(\tilde{\varrho}_\varepsilon,\tilde{\varrho}_\varepsilon\tilde{\mathbf{u}}_\varepsilon)  -g_k(1,\tilde{\mathbf{u}})   \right\vert^2 \right)^\frac{1}{2}\,\mathrm{d}x
\end{align*}
where $\int_{\mathrm{supp}(\phi)}f\,\mathrm{d}x$ is the restriction of the integral of $f$ to the support of $\phi$.

Now let  $\mathbf{x}:=(\tilde{\varrho}_\varepsilon,\tilde{\varrho}_\varepsilon\tilde{\mathbf{u}}_\varepsilon)$ and $\mathbf{y}:=(1,\tilde{\mathbf{u}})$ be vectors in $\mathbb{R}^{4}$  and define the line segment joining them by
\begin{align*}
L(\mathbf{x},\mathbf{y})  =\{t \mathbf{x}  +(1+t)\mathbf{y}   \,:\, 0\leq t\leq 1  \}.
\end{align*}
Then by the Mean value inequality, we can find $(\underline{\varrho}_\varepsilon , \underline{\mathbf{q}}_\varepsilon) \in L(\mathbf{x},\mathbf{y})$ such that
\begin{align*}
& \int_{\mathrm{supp}(\phi)}\left( \sum_{k\in\mathbb{N}}\left\vert  g_k(\tilde{\varrho}_\varepsilon,\tilde{\varrho}_\varepsilon\tilde{\mathbf{u}}_\varepsilon)  -g_k(1,\tilde{\mathbf{u}})   \right\vert^2 \right)^\frac{1}{2}\,\mathrm{d}x
\\
&\leq 
\int_{\mathrm{supp}(\phi)}\left( \left\vert (\tilde{\varrho}_\varepsilon,\tilde{\varrho}_\varepsilon\tilde{\mathbf{u}}_\varepsilon)-(1,\tilde{\mathbf{u}})\right\vert^2 \sum_{k\in\mathbb{N}}\left\vert  \nabla_{\underline{\varrho}_\varepsilon , \underline{\mathbf{q}}_\varepsilon}g_k(\underline{\varrho_\varepsilon} , \underline{\mathbf{q}_\varepsilon})  \right\vert^2 \right)^\frac{1}{2}\,\mathrm{d}x
\\
&\leq c\,
\left(
 \int_{\mathrm{supp}(\phi)} \left\vert \tilde{\varrho}_\varepsilon - 1\right\vert\,\mathrm{d}x
 +
 \int_{\mathrm{supp}(\phi)}
  \left\vert \tilde{\varrho}_\varepsilon \tilde{\mathbf{u}}_\varepsilon - \tilde{\mathbf{u}}\right\vert 
\,\mathrm{d}x \right)
\\
&=: I_1  + I_2
\end{align*}
where we have used \eqref{stochCoeffBound} and \cite[Eq. 6.13.6]{hardy1952inequalities} in the last inequality.

Hence by using the embeddings $L^{\min\{2,\gamma\}}\hookrightarrow L^1$ and $L^r\hookrightarrow L^1$, which holds true for any compact set or ball in $\mathbb{R}^3$ and where $r$ is as defined in Lemma \ref{momentumStrong}, we get that $I_1\rightarrow0$ and $I_2\rightarrow0$ for a.e. $(\omega,t)$ in $\tilde{\Omega}\times(0,T)$. This is due to \eqref{strongDens} and  \eqref{strongMomemConv1}. Hence 
\begin{align*}
 \langle \Phi(\tilde{\varrho}_\varepsilon, \tilde{\varrho}_\varepsilon\tilde{\mathbf{u}}_\varepsilon)\cdot, \phi\rangle  \, \rightarrow\,  \langle \Phi(1,\tilde{\mathbf{u}})\cdot,\phi \rangle  \quad\text{in}\quad L_2(\mathfrak{U};\mathbb{R})\quad \tilde{\mathbb{P}}\times \mathcal{L}-a.e.
\end{align*}
which implies that
\begin{align*}
N(\tilde{\varrho}_\varepsilon,   \tilde{\varrho}_\varepsilon\tilde{\mathbf{u}}_\varepsilon)_{t} \, \rightarrow\,  N(1,    \tilde{\mathbf{u}})_{t}  \quad\text{in}\quad L_2(\mathfrak{U};\mathbb{R})\quad \tilde{\mathbb{P}}\times \mathcal{L}-a.e.
\end{align*}
Similar argument holds for $N_k(\tilde{\varrho}_\varepsilon,   \tilde{\varrho}_\varepsilon\tilde{\mathbf{u}}_\varepsilon)_{t}  \rightarrow    N_k(1,    \tilde{\mathbf{u}})_{t}  \quad \tilde{\mathbb{P}}-a.s.$
\end{proof}

Using Lemmata \ref{lem:M} and Lemma \ref{lem:N},  we can now pass to the limit in equation \cite[Eq. 3.14-3.16]{breit2015incompressible} to get that :
\begin{equation}
\begin{aligned}
\label{cont1}
\tilde{\mathbb{E}}\, h(  \textbf{r}_s\tilde{\mathbf{u}},  \textbf{r}_s\tilde{W})\left[ M(1,  \tilde{\mathbf{u}},  \tilde{\mathbf{u}})_{s,t}  \right]
=0,
\\
\tilde{\mathbb{E}}\, h(  \textbf{r}_s\tilde{\mathbf{u}},  \textbf{r}_s\tilde{W})\left[\left[  M(1,  \tilde{\mathbf{u}},  \tilde{\mathbf{u}})^2\right]_{s,t}   -  N(1,  \tilde{\mathbf{u}})_{s,t} \right]
=0,
\\
\tilde{\mathbb{E}}\, h(  \textbf{r}_s\tilde{\mathbf{u}},  \textbf{r}_s\tilde{W})\left[\left[  M(1,  \tilde{\mathbf{u}},  \tilde{\mathbf{u}})\tilde{\beta}_k\right]_{s,t}   - N(1,  \tilde{\mathbf{u}})_{s,t} \right]
=0.
\end{aligned}
\end{equation}
Equation \eqref{cont1} means that $M(1,  \tilde{\mathbf{u}},  \tilde{\mathbf{u}})_t$ is an $(\mathscr{F}_t)-$martingale. Moreover, using \eqref{cont1}$_2$, we get the quadratic and cross-variation of $M(1,  \tilde{\mathbf{u}},  \tilde{\mathbf{u}})_t$ as
\begin{align*}
\big\langle\big\langle  M(1,  \tilde{\mathbf{u}},  \tilde{\mathbf{u}})_t   \big\rangle\big\rangle= N(1, \tilde{\mathbf{u}}),
\\
\big\langle\big\langle M(1,  \tilde{\mathbf{u}},  \tilde{\mathbf{u}})_t , \tilde{\beta}_k  \big\rangle\big\rangle= N_k(1, \tilde{\mathbf{u}})
\end{align*}
which yields
\begin{align*}
\Big\langle\Big\langle M(1,  \tilde{\mathbf{u}},  \tilde{\mathbf{u}})_t  -   \int_0^t\langle \Phi(1, \tilde{\mathbf{u}})\,\mathrm{d}\tilde{W}, \phi  \rangle \Big\rangle\Big\rangle = 0.
\end{align*}
That is, for $\phi\in C^\infty_{c,\mathrm{div}}(\mathbb{R}^3)$ and $t\in [0,T]$, we have that
\begin{align*}
\langle  \tilde{\mathbf{u}}(t), \phi \rangle   =  \langle \tilde{\mathbf{u}}(0), \phi \rangle  +  \int_0^t  \langle \tilde{\mathbf{u}}\otimes \tilde{\mathbf{u}}, \nabla\phi \rangle\mathrm{d}s  -\nu \int_0^t  \langle \nabla \tilde{\mathbf{u}}, \nabla\phi \rangle\mathrm{d}s      +   \int_0^t\langle \Phi(1, \tilde{\mathbf{u}})\,\mathrm{d}\tilde{W}, \phi  \rangle 
\end{align*}
$\tilde{\mathbb{P}}-$a.s. keeping in mind that $\mathrm{div}\phi=0$. 

\end{proof}

\bibliographystyle{acm}

\begin{thebibliography}{10}

\bibitem{bensoussan1973equations}
{\sc Bensoussan, A., and Temam, R.}
\newblock \'{E}quations stochastiques du type {N}avier--{S}tokes.
\newblock {\em J. Funct. Anal. 13\/} (1973), 195--222.

\bibitem{borchers1990equations}
{\sc Borchers, W., and Sohr, H.}
\newblock On the equations {${\rm rot}\,{\bf v}={\bf g}$} and {${\rm div}\,{\bf
  u}=f$} with zero boundary conditions.
\newblock {\em Hokkaido Math. J. 19}, 1 (1990), 67--87.

\bibitem{breit2015incompressible}
{\sc Breit, D., Feireisl, E., and Hofmanov{\'a}, M.}
\newblock Incompressible limit for compressible fluids with stochastic forcing.
\newblock {\em Arch. Ration. Mech. Anal. 222}, 2 (2016), 895--926.

\bibitem{breit2015compressible}
{\sc Breit, D., Feireisl, E., and Hofmanov{\'a}, M.}
\newblock Compressible fluids driven by stochastic forcing: The relative energy
  inequality and applications.
\newblock {\em Comm. Math. Phys. 350}, 2 (2017), 443--473.

\bibitem{Hof}
{\sc Breit, D., and Hofmanov\'{a}, M.}
\newblock Stochastic {N}avier--{S}tokes equations for compressible fluids.
\newblock {\em Indiana Univ. Math. J. 65}, 4 (2016), 1183--1250.

\bibitem{capinski1994stochastic}
{\sc Capinski, M., and Gatarek, D.}
\newblock Stochastic equations in hilbert space with application to
  {N}avier--{S}tokes equations in any dimension.
\newblock {\em J. Funct. Anal. 126}, 1 (1994), 26--35.

\bibitem{da2014stochastic}
{\sc Da~Prato, G., and Zabczyk, J.}
\newblock {\em Stochastic equations in infinite dimensions}, vol.~152.
\newblock Cambridge university press, 2014.

\bibitem{desjardins1999low}
{\sc Desjardins, B., and Grenier, E.}
\newblock Low mach number limit of viscous compressible flows in the whole
  space.
\newblock In {\em Proceedings of the Royal Society of London A: Mathematical,
  Physical and Engineering Sciences\/} (1999), vol.~455, The Royal Society,
  pp.~2271--2279.

\bibitem{diening2010decomposition}
{\sc Diening, L., Ruzicka, M., and Schumacher, K.}
\newblock A decomposition technique for john domains.
\newblock {\em Ann. Acad. Sci. Fenn. Math. 35}, 1 (2010), 87--114.

\bibitem{farwig1994generalized}
{\sc Farwig, R., and Sohr, H.}
\newblock Generalized resolvent estimates for the {S}tokes system in bounded
  and unbounded domains.
\newblock {\em J. Math. Soc. Jpn. 46}, 4 (1994), 607--643.

\bibitem{feireisl2001compactness}
{\sc Feireisl, E.}
\newblock On compactness of solutions to the compressible isentropic
  {N}avier--{S}tokes equations when the density is not square integrable.
\newblock {\em Comment. Math. Univ. Carol. 42}, 1 (2001), 83--98.

\bibitem{feireisl2012multi}
{\sc Feireisl, E., Gallagher, I., Gerard-Varet, D., and Novotn\'y, A.~n.}
\newblock Multi-scale analysis of compressible viscous and rotating fluids.
\newblock {\em Comm. Math. Phys. 314}, 3 (2012), 641--670.

\bibitem{feireisl2013compressible}
{\sc Feireisl, E., Maslowski, B., and Novotn{\'y}, A.}
\newblock Compressible fluid flows driven by stochastic forcing.
\newblock {\em J. Differential Equations 254}, 3 (2013), 1342--1358.

\bibitem{feireisl2009singular}
{\sc Feireisl, E., and Novotn{\'y}, A.}
\newblock {\em Singular limits in thermodynamics of viscous fluids}.
\newblock Springer Science \& Business Media, 2009.

\bibitem{feireisl2001existence}
{\sc Feireisl, E., Novotn{\'y}, A., and Petzeltov{\'a}, H.}
\newblock On the existence of globally defined weak solutions to the
  {N}avier--{S}tokes equations.
\newblock {\em J. Math. Fluid Mech. 3}, 4 (2001), 358--392.

\bibitem{flandoli1995martingale}
{\sc Flandoli, F., and Gatarek, D.}
\newblock Martingale and stationary solutions for stochastic {N}avier--{S}tokes
  equations.
\newblock {\em Probab. Theory Relat. Fields 102}, 3 (1995), 367--391.

\bibitem{galdi2011introduction}
{\sc Galdi, G.~P.}
\newblock {\em An introduction to the mathematical theory of the
  {N}avier--{S}tokes equations}, second~ed.
\newblock Springer Monographs in Mathematics. Springer, New York, 2011.
\newblock Steady-state problems.

\bibitem{geissert2006equation}
{\sc Gei\ss~ert, M., Heck, H., and Hieber, M.}
\newblock On the equation {${\rm div}\,u=g$} and {B}ogovski\u\i 's operator in
  {S}obolev spaces of negative order.
\newblock In {\em Partial differential equations and functional analysis},
  vol.~168 of {\em Oper. Theory Adv. Appl.} Birkh\"auser, Basel, 2006,
  pp.~113--121.

\bibitem{hardy1952inequalities}
{\sc Hardy, G.~H., Littlewood, J.~E., and P{\'o}lya, G.}
\newblock {\em Inequalities}.
\newblock Cambridge university press, 1952.

\bibitem{jakubowski1998short}
{\sc Jakubowski, A.}
\newblock Short communication: The almost sure skorokhod representation for
  subsequences in nonmetric spaces.
\newblock {\em Theory Probab. Appl. 42}, 1 (1998), 167--174.

\bibitem{klainerman1981singular}
{\sc Klainerman, S., and Majda, A.}
\newblock Singular limits of quasilinear hyperbolic systems with large
  parameters and the incompressible limit of compressible fluids.
\newblock {\em Commun. Pure Appl. Math. 34}, 4 (1981), 481--524.

\bibitem{leray1934mouvement}
{\sc Leray, J.}
\newblock Sur le mouvement d'un liquide visqueux emplissant l'espace.
\newblock {\em Acta Math. 63}, 1 (1934), 193--248.

\bibitem{lions1998mathematical}
{\sc Lions, P.-L.}
\newblock {\em Mathematical topics in fluid mechanics. {V}ol. 2}, vol.~10 of
  {\em Oxford Lecture Series in Mathematics and its Applications}.
\newblock The Clarendon Press, Oxford University Press, New York, 1998.
\newblock Compressible models, Oxford Science Publications.

\bibitem{lions1998incompressible}
{\sc Lions, P.-L., and Masmoudi, N.}
\newblock Incompressible limit for a viscous compressible fluid.
\newblock {\em J. Math. Pures Appl. (9) 77}, 6 (1998), 585--627.

\bibitem{lions1998unicite}
{\sc Lions, P.-L., and Masmoudi, N.}
\newblock Unicit\'e des solutions faibles de {N}avier--{S}tokes dans
  {$L^N(\Omega)$}.
\newblock {\em C. R. Acad. Sci. Paris S\'er. I Math. 327}, 5 (1998), 491--496.

\bibitem{lions1999approche}
{\sc Lions, P.-L., and Masmoudi, N.}
\newblock Une approche locale de la limite incompressible.
\newblock {\em C. R. Acad. Sci. Paris S\'er. I Math. 329}, 5 (1999), 387--392.

\bibitem{mikulevicius2005global}
{\sc Mikulevicius, R., and Rozovskii, B.~L.}
\newblock Global {$L_2$}-solutions of stochastic {N}avier--{S}tokes equations.
\newblock {\em Ann. Probab. 33}, 1 (2005), 137--176.

\bibitem{straskraba2004introduction}
{\sc Novotn\'{y}, A., and Stra\u{s}kraba, I.}
\newblock {\em Introduction to the mathematical theory of compressible flow}.
\newblock Oxford University Press, New York, 2004.

\bibitem{robinson2016recent}
{\sc Robinson, J.~C., Rodrigo, J.~L., Sadowski, W., and Vidal-L{\'o}pez, A.},
  Eds.
\newblock {\em Recent progress in the theory of the {E}uler and
  {N}avier--{S}tokes equations}, vol.~430 of {\em London Mathematical Society
  Lecture Note Series}.
\newblock Cambridge University Press, Cambridge, 2016.
\newblock Including paper from the workshop ``The {N}avier--{S}tokes Equations
  in Venice'' held in Venice, April 8--12, 2013.

\bibitem{romito2016probabilistic}
{\sc Romito, M.}
\newblock Some probabilistic topics in the {N}avier--{S}tokes equations.
\newblock In {\em Recent progress in the theory of the {E}uler and
  {N}avier--{S}tokes equations}, vol.~430 of {\em London Mathematical Society
  Lecture Note Series}. Cambridge University Press, Cambridge, 2016,
  pp.~175--232.

\bibitem{smith2000global}
{\sc Smith, H.~F., and Sogge, C.~D.}
\newblock Global {S}trichartz estimates for nontrapping perturbations of the
  {L}aplacian.
\newblock {\em Commun. Partial Differ. Equ. 25}, 11-12 (2000), 2171--2183.

\bibitem{smith2015random}
{\sc Smith, S.}
\newblock Random perturbations of viscous compressible fluids: Global existence
  of weak solutions.
\newblock {\em arXiv preprint arXiv:1504.00951\/} (2015).

\bibitem{tornatore2000global}
{\sc Tornatore, E.}
\newblock Global solution of bi-dimensional stochastic equation for a viscous
  gas.
\newblock {\em NoDEA Nonlinear Differential Equations Appl. 7}, 4 (2000),
  343--360.

\bibitem{tornatore1997one}
{\sc Tornatore, E., and Fujita~Yashima, H.}
\newblock One-dimensional stochastic equations for a viscous barotropic gas.
\newblock {\em Ric. Mat. 46}, 2 (1997), 255--283.

\bibitem{vaigant1995existence}
{\sc Va{\u\i}gant, V.~A., and Kazhikhov, A.~V.}
\newblock On the existence of global solutions of two-dimensional
  {N}avier--{S}tokes equations of a compressible viscous fluid.
\newblock {\em Dokl. Akad. Nauk 357}, 4 (1997), 445--448.

\end{thebibliography}

\end{document}